\documentclass[12pt,a4paper,leqno]{amsart}

\usepackage[top=1.1in, bottom=1.3in, left=1.5in, right=1.5in]{geometry}
\usepackage{slashbox}
\usepackage[ansinew]{inputenc}
\usepackage[T1]{fontenc}
\usepackage[english]{babel}
\usepackage{amsthm}
\usepackage{amsfonts}         
\usepackage{amsmath}
\usepackage{amssymb}
\usepackage{breqn}
\usepackage{bm}
\usepackage{url}
\usepackage{esint}
\usepackage{fancyhdr}

\newcommand{\R}{\mathbb{R}}
\newcommand{\C}{\mathbb{C}}
\newcommand{\Q}{\mathbb{Q}}
\newcommand{\N}{\mathbb{N}}

\newcommand{\Z}{\mathbb{Z}}

\newcommand{\A}{\mathcal{A}}

\newcommand\pprec{\prec\mkern-5mu\prec}

\newcommand{\K}{\mathcal{K}}

\renewcommand{\a}{\mathfrak{a}}

\renewcommand{\d}{\mathfrak{d}}
\newcommand{\m}{\mathfrak{m}}
\newcommand{\n}{\mathfrak{n}}
\newcommand{\p}{\mathfrak{p}}

\newcommand{\sumw}{\sideset{}{^\wedge}\sum}
\renewcommand{\epsilon}{\varepsilon}

\theoremstyle{plain}
\newtheorem{theorem}{Theorem}
\newtheorem{lemma}[theorem]{Lemma}
\newtheorem{prop}[theorem]{Proposition}
\newtheorem{cor}[theorem]{Corollary}
\newtheorem{definition}{Definition}

\theoremstyle{remark}
\newtheorem{remark}{Remark}

\numberwithin{equation}{section}
\begin{document}
\title{A Cubic analogue of the Friedlander-Iwaniec spin over primes}
\author{Jori Merikoski}
\address{Mathematical Institute,
University of Oxford,
Andrew Wiles Building,
Radcliffe Observatory Quarter,
Woodstock Road,
Oxford,
OX2 6GG}
\email{jori.merikoski@maths.ox.ac.uk}

\begin{abstract} In 1998 Friedlander and Iwaniec proved that there are infinitely many primes of the form $a^2+b^4$. To show this they used the Jacobi symbol to define the spin of Gaussian integers, and one of the key ingredients in the proof was to show that the spin becomes equidistributed along Gaussian primes. To generalize this we define the cubic spin of ideals of $\Z[\zeta_{12}]=\Z[\zeta_3,i]$ by using the cubic residue character on the Eisenstein integers $\Z[\zeta_3]$. Our main theorem says that the cubic spin is equidistributed along prime ideals of $\Z[\zeta_{12}]$. The proof of this follows closely along the lines of Friedlander and Iwaniec. The main new feature in our case is the infinite unit group, which means that we need to show that the definition of the cubic spin on the ring of integers lifts to a well-defined function on the ideals. We also explain how the cubic spin arises if we consider primes of the form $a^2+b^6$ on the Eisenstein integers.
\end{abstract}

\maketitle

\tableofcontents
\section{Introduction}
Friedlander and Iwaniec \cite{FI} famously showed that there are infinitely many prime numbers represented by $a^2+b^4$. Remarkable here is that numbers of this form are very sparse, that is, the number of such integers up to $x$ is of order $x^{3/4}$. Other similar results are Heath-Brown's proof that there are infinitely many primes of the form $a^3+2b^3$ \cite{hb}, the generalization of Heath-Brown and Moroz of this to binary cubic forms \cite{hbm}, the extension of this by Maynard to general incomplete norm forms \cite{maynard}, and the result of Heath-Brown and Li \cite{hbli} that the Friedlander-Iwaniec result holds also with $b$ restricted to prime values.

If a prime $p$ is of the form $a^2+b^4$, then $p=\pi \overline{\pi}$ for some Gaussian prime $\pi=b^2+ia$, so that the arithmetic in the work of Friedlander and Iwaniec really lies in $\Z[i]$. For a Gaussian integer $z=r+is$ with $r$ odd, define the quadratic spin as
\begin{align*}
[z]_2 := \bigg( \frac{s}{r}\bigg)_2,
\end{align*}
where $(s/r)_2$ is the Jacobi symbol. One of the key ingredients in the proof of Friedlander and Iwaniec is to show that the spin is equidistributed along Gaussian primes \cite[Theorem 2]{FI}, which they obtained in the form
\begin{align} \label{gaussspin}
\sum_{\substack{p=r^2+s^2 \leq x \\ 2 \, \nmid \, r}} \bigg( \frac{s}{r}\bigg)_2 \ll x^{1-1/77}.
\end{align}
This has been generalized by Milovic to show equidistribution of $(v/u)_2$ over primes of the form $p=u^2-2 v^2$, which corresponds to the above with $\Z[\sqrt{2}]$ in place of $\Z[i]$ \cite[Theorem 2]{milovic}.

It is natural to ask if the argument can be extended to produce primes of the form $a^2+b^6$. Friedlander and Iwaniec have solved the ternary divisor problem for this sequence \cite{fidivisor}, and under the assumption of the existence of exceptional Dirichlet characters they have shown that there are infinitely many primes of this form \cite{fiillusory}.

At the moment there seems to be two major obstacles to solving the problem of primes of the form $a^2+b^6$. Firstly, the sequence of integers is now too sparse to replicate the steps in \cite[Sections 5-9]{FI}. The second problem is structural. Recall that the proof of (\ref{gaussspin}) relies on the law of quadratic reciprocity. With the sequence $a^2+b^6$ we end up with cubic residues which unfortunately do not satisfy a suitable reciprocity law on $\Z$.

The second obstacle can be overcome if we extend the whole set-up from $\Z$ to the Eisenstein integers $\Z[\zeta_3]$, where the cubic residue character does satisfy a reciprocity law (see Lemma \ref{cubicreciprocitylemma}).  Unfortunately the first issue remains and we are not able to detect primes of $\Z[\zeta_3]$ of the form $a^2+b^6$ with $a,b \in \Z[\zeta_3]$ (see  Section \ref{primessection} for more details). However, we can still obtain the analogue of (\ref{gaussspin}) in this situation and thus make partial progress on this problem. The Gaussian integers now correspond to the ring $\Z[\zeta_{12}]= \Z[\zeta_3,i]$ of integers of the twelfth cyclotomic extension, since the relative norm is $N_{\Q(\zeta_{12})/\Q(\zeta_{3})}(r+is) = r^2+s^2$ for $r,s \in \Q(\zeta_{3})$.

We say that $z \in \Z[\zeta_{12}]$ is primary if $z \equiv \pm 1 \, (\text{mod} \, 3)$. For any $(z,3)=1$ there exists a unit $\mu$ such that $\mu z$ is primary. For a primary number $z=r+is \in \Z[\zeta_{12}]$ we define the cubic spin
\begin{align*}
[z]_3 := \bigg[ \frac{s}{r}\bigg]_3,
\end{align*}
where $[s/r]_3$ is the cubic residue character on $\Z[\zeta_3]$ (see Section \ref{cubicsection} for details). We extend this to the ideals $\a$ of $\Z[\zeta_{12}]$ by defining $[\a]_3 := [z]_3$ if $z$ is a primary generator of $\a$ with $(r,s)=1$ and set $[\a]_3=0$ otherwise. In Section \ref{12section} we will show that this definition does not depend on the choice of the primary associate $z$  (note that by Dirichlet's unit theorem there are infinitely many possible choices). Our main theorem says that the values of the cubic spin are equidistributed along prime ideals of $\Z[\zeta_{12}]$.
\begin{theorem} \label{maintheorem}
We have
\begin{align*}
\sum_{N_{\Q(\zeta_{12})} \,\p \leq x} [\p]_3 \, \ll \, x^{1-1/143}.
\end{align*}
\end{theorem}
Similarly as in \cite[Theorem 2]{FI}, the exponent 1/143 is not the best that could be obtained and we have opted for simplicity in the proof over optimality.

The above cubic spin and the spin of Friedlander and Iwaniec \cite{FI} should not be confused with the spin of a prime ideal as defined by Friedlander, Iwaniec, Mazur, and Rubin \cite{fimr}.

Note that for a given prime $\p=(r+is)$ all of its Galois conjugates appear in the sum, which means that the sum is real. Indeed, for $r+is$ primary we have
\begin{align*}
\sum_{\sigma \in \text{Gal}(\Q(\zeta_{12})/\Q)} [\sigma (r+is)]_3 = \bigg[ \frac{s}{r}\bigg]_3+ \bigg[ \frac{-s}{r}\bigg]_3 +\bigg[ \frac{\bar{s}}{\bar{r}}\bigg]_3+\bigg[ \frac{-\bar{s}}{\bar{r}}\bigg]_3 = 4 \,  \text{Re} \bigg( \bigg[ \frac{s}{r}\bigg]_3 \bigg)
\end{align*} 
by using the properties $[\bar{s}/\bar{r}]_3 = \overline{[s/r]}_3$ and $[-s/r]_3=[s/r]_3$.

Note also that for $\pi = r^2+s^2$ with  $r+is$ primary we have by cubic reciprocity (Lemma \ref{cubicreciprocitylemma})
\begin{align} \label{rcube}
\bigg[\frac{s}{r}\bigg]=\bigg[\frac{s^2}{r}\bigg]^2=\bigg[\frac{\pi}{r}\bigg]^2 = \bigg[\frac{r}{\pi}\bigg]^2,
\end{align}
since $r+is$ being primary implies that $r$ and $\pi$ are primary in $\Z[\zeta_3]$. Thus, our main theorem implies that $r$ is a cube modulo primes $\pi$ asymptotically one third of the time (to prove this, expand $1_{r\equiv t^3 \, \pmod \pi} = (1+[r/\pi] + [r/\pi]^2)/3$ and note that $\sum [\p]_3^2 = \overline{\sum [\p]_3}$).
\begin{cor} \label{cubecorollary} We have
\begin{align*}
\sum_{ \substack{N_{\Q(\zeta_3)} \pi \leq x \\ \pi=r^2+s^2, \\r \equiv \pm 1 \, \pmod 3, \,\, 3|s}} 1_{r \equiv t^3 \, \pmod \pi } = \frac{1}{3}\sum_{ \substack{N_{\Q(\zeta_3)} \pi \leq x \\ \pi=r^2+s^2, \\r \equiv \pm 1 \, \pmod 3, \,\, 3|s}} 1 + O(x^{1-1/143}).
\end{align*}
\end{cor}

\begin{remark} Since we prove that $[\a]_3$ is independent of the choice of the primary generator $z=r+is$, this means by (\ref{rcube}) that for primes $\pi=r^2+s^2$ the property that $r \equiv t^3 \, \pmod \pi$ is independent of the representation $\pi=r^2+s^2$ where $r+is$ is primary, so that the sum in the above corollary is well-defined. The corollary may be viewed as an approximation to the problem of primes of the form $t^6+s^2$ on $\Z[\zeta_3]$ -- instead of $r$ being a perfect cube, it is a cube modulo $\pi=r^2+s^2$.
\end{remark}

For any integer $n \geq 1$ define
\begin{align*}
\lambda_3(n):= \sum_{N_{\Q(\zeta_{12})} \,\a = n} [\a]_3.
\end{align*}
For rational primes we get the following corollary of Theorem \ref{maintheorem} (the error term we get from the proof of Theorem \ref{maintheorem} is actually $O_\epsilon(x^{1-1/142+\epsilon})$ so that the same error term holds for the corollary below).
\begin{cor}
We have
\begin{align*}
\sum_{n \leq x} \Lambda(n) \lambda_3(n)  \ll \, x^{1-1/143}.
\end{align*}
\end{cor}

\begin{remark} \label{genremark} The proof of Theorem  \ref{maintheorem} relies mainly on the law of cubic reciprocity. Thus, it seems plausible that the result can be generalized as follows. If an algebraic number field $K$ contains a primitive $m$th root of unity, then we can define the $m$th power residue character on $K$ which satisfies a reciprocity law (see \cite[Chapter VIII, Theorem 5.11]{milne}, for instance). Given a quadratic extension $L/K$ we can then define a spin at elements of $\mathcal{O}_L$. It seems plausible that the argument could be generalized to obtain equidistribution of the spin along principal prime ideals of $L$. We hope to attack this question in a future article. Probably some assumptions are required here. At least in the simplest case of $K=\Q(\zeta_m)$ with $m$ odd prime and $L=K[i]$ many parts of the argument seem to generalize nicely.
\end{remark}

\begin{remark} In \cite[Section 23]{FI} Friedlander and Iwaniec write "We suspect, but have not examined thoroughly, that $\lambda(n)$ are
related to the Fourier coefficients of some kind of metaplectic Eisenstein series or a cusp form, by analogy to the Hecke eigenvalues (16.30) which generate a modular form of integral weight." Similarly, we expect that $\lambda_3(n)$ can be interpreted in terms of automorphic forms. Working out the details of this would be useful with a view towards the generalization outlined in the previous remark.
\end{remark}

\begin{remark} We suspect that our main theorem has some applications to elliptic curves over $\Z[\zeta_3]$ but we do not have anything particularly interesting. For example, if $\pi=r^2+s^2$ is a prime and $r \equiv t^3 \, \pmod \pi$, then Corollary \ref{cubecorollary} provides us with a large family of elliptic curves $E: \,Y^2=X^3+3t^2 X  \pm 2s$ with bad reduction at some large prime $\pi=\pi(E)$.
\end{remark}
\subsection{A brief sketch of the Friedlander-Iwaniec argument}
We present here a non-rigorous sketch of the proof of (\ref{gaussspin}) which appears in \cite[Sections 19-26]{FI}. Recall that the claim is that
\begin{align*}
\sum_{\substack{ z \in \Z[i] \\|z|^2=p \leq x}} [z]_2 \ll x^{1-1/77},
\end{align*}
where $[r+is]_2 = (s/r)_2$ is the usual Jacobi symbol. The summation needs to be restricted to odd $r$ but let us ignore this in the notation to simplify the presentation. Then by a sieve argument (essentially Vaughan's identity) the task is reduced to bounding Type I sums
\begin{align} \label{typeifi}
\sum_{\substack{ |w|^2 \sim M}}\alpha_w \sum_{\substack{ |z|^2 \sim N}} [wz]_2 
\end{align}
and Type II sums
\begin{align} \label{typeiifi}
\sum_{\substack{ |w|^2 \sim M}}\alpha_w \sum_{\substack{ |z|^2 \sim N}} \beta_z [wz]_2,
\end{align}
where $\alpha_w$ and $\beta_z$ are arbitrary bounded coefficients and $MN=x$ with sufficiently flexible ranges for $M$ and $N$.

For $w=u+iv$ with $(u,v)=1$, let $\omega \equiv - vu^{-1} \,\, \pmod{u^2+v^2}$, where $u^{-1}$ denotes the multiplicative inverse, so that $\omega^2 \equiv -1 \,\, \pmod{u^2+v^2}$. Similarly as in \cite[Section 19]{FI}, for any $z=r+is$ we define
\begin{align*}
\bigg(\frac{z}{w}\bigg)_2:=\bigg(\frac{r+\omega s}{u^2+v^2}\bigg)_2.
\end{align*} 
Since $\omega^2 \equiv -1 \,\, \pmod{u^2+v^2}$, this symbol is completely multiplicative in the upper variable and is therefore a quadratic character modulo $u^2+v^2$. We also have $(z/w_1)_2(z/w_2)_2= (z/w_1 w_2)_2$ provided that $(w_1,\overline{w_2})=1$. The general multiplicativity rule in the lower part is not much more complicated, but to simplify let us pretend that the symbol is completely multiplicative also in the lower varliable.

The key lemma is \cite[Lemma 20.1]{FI}, which morally states that the quadratic spin is twisted multiplicative in the sense that
\begin{align*}
[wz]_2 = \mathcal{E} [w]_2[z]_2 \bigg(\frac{z}{w}\bigg)_2
\end{align*}
for some simple sign factor $\mathcal{E}$. The proof of this relies on quadratic reciprocity multiple times. To simplify the presentation we pretend that this holds with $\mathcal{E}=1$. Then bounding the Type I and Type II sums is reduced to (absorbing factors into the coefficients $\alpha_w$ and $\beta_w$)
\begin{align*} 
\sum_{\substack{ |w|^2 \sim M}}\alpha_w \sum_{\substack{ |z|^2 \sim N}} [z]_2\bigg(\frac{z}{w}\bigg)_2 \quad \quad \text{and} \quad \quad
\sum_{\substack{ |w|^2 \sim M}}\alpha_w \sum_{\substack{ |z|^2 \sim N}} \beta_z \bigg(\frac{z}{w}\bigg)_2.
\end{align*}

For the Type I sums (see \cite[Section 22]{FI}) we fix $w$ and write
\begin{align*}
\bigg|\sum_{\substack{ |z|^2 \sim N}} [z]_2\bigg(\frac{z}{w}\bigg)_2\bigg| \leq \sum_{r \ll \sqrt{N}} \bigg| \sum_{\substack{ s^2 \sim N-r^2}} \bigg[\frac{s}{r}\bigg]_2\bigg(\frac{r+s\omega}{u^2+v^2}\bigg)_2 \bigg| \\
=\sum_{r \ll \sqrt{N}} \bigg| \sum_{\substack{ s  \in I(r)}} \bigg(\frac{s}{r(u^2+v^2)}\bigg)_2 \bigg|
\end{align*}
by making making the change of variables $s \mapsto s + \omega r$, where $I(r)$ denotes an interval of length $\ll \sqrt{N}$. The sum over $s$ can be bounded using the P\'olya-Vinogradov inequality for short character sums, which yields a non-trivial bound for the Type I sums (\ref{typeifi}) in the range $M \leq x^{1/3-\eta}$ for any $\eta>0$.

To handle the Type II sums (see \cite[Section 21]{FI}) we use Cauchy-Schwarz to morally get
\begin{align*}
\sum_{\substack{ |w|^2 \sim M}}\alpha_w \sum_{\substack{ |z|^2 \sim N}} \beta_z \bigg(\frac{z}{w}\bigg)_2 \ll N^{1/2} \bigg( \sum_{\substack{ |w_1|^2,  |w_2|^2\sim M}}\alpha_{w_1}\alpha_{w_2} \sum_{\substack{ |z|^2 \sim N}} \bigg(\frac{z}{w_1 w_2}\bigg)_2\bigg)^{1/2}.
\end{align*}
Since $(z/w_1w_2)_2$ is a quadratic character modulo $|w_1w_2|^2$, the sum over $z$ is very small unless $|w_1w_2|^2$ is a perfect square (at least for $N$ large compared to $|w_1w_2|^2 \asymp M^2$). The part where $|w_1w_2|^2$ is a perfect square is a very narrow subset of the variables, which gives a non-trivial bound for the Type II sums. This bound can be amplified by a suitable application of H\"older's inequality and by the reciprocity $(z/w)_2=(w/z)_2$. We get a non-trivial bound for the Type II sums (\ref{typeiifi}) in the full range $x^{\eta} \ll M,N \ll x^{1-\eta}$.

\subsection{Structure of the article}
The proof of Theorem \ref{maintheorem} follows the same lines as the proof of \cite[Theorem 2]{FI}. In Section \ref{cubicsection} we recall the law of cubic reciprocity and prove that the cubic spin $[z]_3$ satisfies a twisted multiplicativity relation (Lemma \ref{twistmultiplemma}). This relation is the key ingredient in all of the arguments that follow. In Section \ref{12section} we recall basic facts about $\Q(\zeta_{12})$ and show that the definition of the spin $[\a]_3=[z]_3$ is independent of the choice of primary generator $z$ of $\a$ (also for this we need Lemma \ref{twistmultiplemma}). 

In Section \ref{sievesection} we use Buchstab's identity to obtain a decomposition of the prime sum into sums of Type I and Type II. We could also use similar arguments as in \cite{FI} to this end. In Section \ref{fixingsection} we explain how to choose unique primary generators for the ideals of $\a$ in a consistent manner.

After these steps the arugment is essentially same as in \cite{FI} with only minor modifications. In Sections \ref{typeisection} and \ref{typeiisection} we compute the Type I and Type II sums, respectively, which by Section \ref{sievesection} completes the proof of Theorem \ref{maintheorem}. For this we need a version of the Poisson summation on $\Z[\zeta_3]$, which is given in Section \ref{poissonsection}. The reason why the exponent in Theorem \ref{maintheorem} is worse than that in \cite[Theorem 2]{FI} is solely because in the Type I sums we essentially get a contribution from the error term in a lattice point counting problem on $\Z[\zeta_3]$.

Lastly, in Section \ref{primessection} we illustrate non-rigorously how the cupic spin arises from the problem of primes of the type $\alpha^2+\beta^6$ on $\Z[\zeta_3]$. The arguments follow the same lines as in \cite{FI}. We also explain why the density issue prevents us from completing the goal of detecting primes of this form. 
 
\subsection{Notations}
For functions $f$ and $g$, we write $f \ll g$ or $f= \mathcal{O}(g)$ if there is a constant $C$ such that $|f|  \leq C |g|.$ The notation $f \asymp g$ means $g \ll f \ll g.$ The constant may depend on some parameter, which is indicated in the subscript (e.g. $\ll_{\epsilon}$).
We write $f=o(g)$ if $f/g \to 0$ for large values of the variable. For variables we write $n \sim N$ meaning $N<n \leq 2N$. 

For two functions $f$ and $g$ with $g \geq 0$, it is convenient for us to denote  $f(N) \pprec g(N)$ if $f(N) \ll_\epsilon N^{\epsilon}g(N)$. A typical bound we use is $\tau_k(n) \pprec 1$, where $\tau_k$ is the $k$-fold divisor function. For multivariable functions such as sums over two variables we write
\begin{align*}
\sum_{\substack{m \sim M \\ n \sim N}} f(m,n) \pprec \sum_{\substack{m \sim M \\ n \sim N}} g(m,n)
\end{align*}
to mean
\begin{align*}
\sum_{\substack{m \sim M \\ n \sim N}} f(m,n) \ll_\epsilon (M+N)^{\epsilon} \sum_{\substack{m \sim M \\ n \sim N}} g(m,n).
\end{align*}
 We say that an arithmetic function $f$ is divisor bounded if $|f(n)| \ll \tau_k(n)$ for some $k$.

For a statement $E$ we denote by $1_E$ the characteristic function of that statement. For a set $A$ we use $1_A$ to denote the characteristic function of $A.$ 

We let $e(x):= e^{2 \pi i x}$ and $e_q(x):= e(x/q)$ for any integer $q \geq 1$.  We abbreviate modular arithmetic such as $a \equiv b \pmod c$ by $a \equiv b \, (c)$. For any $(a,b)=1$ we let $a^{-1} \, (b)$ denote the multiplicative inverse, so that $a a^{-1} \equiv 1 \, (b)$.

We abbreviate the norm maps as follows. For any $a=a^{(1)}+a^{(2)} \zeta_3 \in \Z[\zeta_3],$ $a^{(j)} \in \Z$, and $\zeta=r+is \in \Z[\zeta_{12}]$, $r,s \in \Z[\zeta_3]$ we set
\begin{align*}
&N_{3}(a):= N_{\Q(\zeta_3)}(a) = (a^{(1)})^2 - a^{(1)} a^{(2)} +  (a^{(2)})^2 = |a|^2, \\
 N_{12/3} (\zeta)&:= N_{\Q(\zeta_{12})/\Q(\zeta_{3})}(\zeta)= r^2+s^2, \quad \text{and} \quad N_{12} := N_{\Q(\zeta_{12})} = N_{3} \circ \N_{12/3}.
\end{align*}

\subsection{Acknowledgements}
I am grateful to my supervisor Kaisa Matom\"aki for helpful comments and support. I also thank Joni Ter\"av\"ainen for comments on an early version of the manuscript. During the work the author was funded by UTUGS Graduate School. Part of the article was also completed while I was working on projects funded by the Academy of Finland (project no. 319180) and the Emil Aaltonen foundation.

\section{Cubic reciprocity} \label{cubicsection}
In this section we recall basic properties of the Eisenstein integers $\Z[\zeta_3]$ and the cubic residue character (cf. \cite[Chapter 7]{lemm}, for instance). We also prove a twisted multiplicativity rule for the cubic spin $[z]_3$ (Lemma \ref{twistmultiplemma}) by using the cubic reciprocity law. To simplify the notation we will abbreviate modular arithmetic such as $a \equiv b \pmod c$ by $a \equiv b \, (c)$.

The unit group of $\Z[\zeta_3]$ is the group of sixth roots of unity $\{\pm 1, \pm \zeta_3, \pm \zeta_3^2\}$. We say that an integer  $a \in \Z[\zeta_3]$ is primary if $a \equiv \pm 1 \, (3)$. Equivalently, $a=a^{(1)} + a^{(2)} \zeta_3$ with $a^{(j)}  \in \Z$ is primary if $3| a^{(2)}$ and $a^{(1)} \equiv \pm 1 \, (3).$ For any $(a,3)=1$ there exists a unit $\mu$ such that $\mu a \equiv 1 \, (3)$.

Any  rational prime $p \equiv 1 \, (3)$ splits as $p=\pi \bar{\pi}$ for a prime $\pi \in \Z[\zeta_3]$. Then for any $a \in \Z[\zeta_3]$, $\pi \nmid a$, we have by Fermat's Little Theorem 
\begin{align*}
a^{p-1} \equiv 1 \,\, (\pi).
\end{align*}
Since $p \equiv 1 \, (3)$, we see that $a^{(p-1)/3} \equiv \zeta_3^k \, (\pi)$ for some $k \in \{0,1,2\}$, so that we may define the cubic residue character modulo $\pi$
\begin{align*}
\bigg[ \frac{a}{\pi} \bigg]_3 := \zeta_3^k.
\end{align*}
If $\pi| a$ we set $[a/\pi]_3 := 0.$
The rational primes $q \equiv 2 \, (3)$ are inert and we define $[a/q]_3:= 1$ if $q \nmid a$ and $[a/q]_3 :=0$ if $q|a$.  For any unit $\mu$ of $\Z[\zeta_3]$ we set $[a/\mu]_3 := 1$, and for the prime $1-\zeta_3$ we set $[a/(1-\zeta_3)]_3 := 1$ (recall that $3=-\zeta_3^2 (1-\zeta_3)^2$ is the only prime that ramifies). Then for any prime $\tau \in \Z[\zeta_3]$ the congruence $x^3 \equiv a \, (\tau)$ has a non-zero solution if and only if $[a/\tau]_3=1$.

For any non-zero $\lambda \in \Z[\zeta_3]$ we have a unique factorization
\begin{align*}
\lambda = \pm \zeta_3^k (1-\zeta_3)^\ell \pi_1^{\alpha_1} \cdots \pi_m^{\alpha_m} q_1^{\beta_1} \cdots q_n^{\beta_n},
\end{align*}
where $\pi_j \equiv 1 \, (3)$, and $q_j \equiv 2 \, (3)$ are rational primes. Therefore, we may extend $[\cdot/\pi]_3$ to all of $\Z[\zeta_3]$ multiplicatively
\begin{align*}
\bigg[ \frac{a}{\lambda} \bigg]_3  := \bigg[ \frac{a}{\pi_1} \bigg]_3^{\alpha_1}  \cdots \bigg[ \frac{a}{\pi_k} \bigg]_3^{\alpha_k}. 
\end{align*}
It is then clear that this is completely multiplicative in both variables. From here on we simplify notations by ignoring the subscript 3, that is, we write $[a/b] := [a/b]_3$.

For any $a,b \in \Z[\zeta_3]$ we write $(a,b)=1$ if $a$ and $b$ are coprime. For any $(a,b)=1$ we let $\epsilon(a,b)$ denote the cubic root of unity such that
\begin{align*}
\bigg[\frac{a}{b}\bigg] =\epsilon(a,b)\bigg[\frac{b}{a}\bigg],
\end{align*}
Note that for any $a,b,c \in \Z[\zeta_3]$ with $(a,bc)=1$ we have multiplicativity in the sense that
\begin{align*}
\epsilon(a,bc)=\epsilon(a,b) \epsilon(a,c) \quad \text{and} \quad \epsilon(bc,a) = \epsilon(b,a)\epsilon(c,a).
\end{align*}
Note also that $\epsilon(a,b)=\epsilon(a,b)^{-1} = \epsilon(b,a)^2$.

By \cite[Theorem 7.8]{lemm} we have the following cubic reciprocity law (which can also be found in \cite[Chapter VIII, Example 5.13]{milne}).
\begin{lemma} \label{cubicreciprocitylemma}\emph{\textbf{(Cubic reciprocity).}} Let $a, b \in \Z[\zeta_{3}]$ be coprime. If $a$ and $b$ are primary, then $\epsilon(a,b)=1$. If $a=a^{(1)}+a^{(2)}\zeta_3$ is primary, then depending on the sign of $a^{(1)} \equiv \pm 1 \,(\text{\emph{mod}} \, 3)$ we have 
\begin{align*}
\epsilon(\zeta_3, a)= \zeta_3^{(1 \pm (-a^{(1)}-a^{(2)}))/3}, \quad \epsilon(1-\zeta_3, a)= \zeta_3^{(\pm a^{(1)}-1)/3}, \quad \epsilon(3,a)= \zeta_3^{\pm a^{(2)}/3},
\end{align*}
so that for any $c \in \Z[\zeta_3]$ we have
\begin{align*}
\epsilon(\zeta_3, a+ 9c)=\epsilon(\zeta_3, a) \quad \text{and} \quad  \epsilon(1-\zeta_3, a+ 9c)=\epsilon(1-\zeta_3, a).
\end{align*}
\end{lemma}

We say that $z \in \Z[\zeta_{12}]$ is primary if $z \equiv \pm 1 \,\, (3)$. For a primary $z=r+is \in \Z[\zeta_{12}]$ we define
\begin{align*}
[z]:= [z]_3=\bigg[ \frac{s}{r}\bigg].
\end{align*}
We will extend this definition to ideals of $\Z[\zeta_{12}]$ in Section \ref{12section}. We say that $w=u+iv \in \Z[\zeta_{12}]$ is primitive if $(u,v)=1.$  For $w \in Z[\zeta_{12}]$ primary primitive, set $\omega \equiv -v u^{-1} \,\, (\text{mod} \, u^2+v^2)$, where $u^{-1}$ is the multiplicative inverse modulo $u^2+v^2$. Analogously to the Dirichlet symbol defined in \cite[Section 19]{FI}, we define
\begin{align} \label{symboldef}
\bigg(\frac{z}{w} \bigg) := \bigg[\frac{r+\omega s}{u^2+v^2} \bigg].
\end{align}
Since $\omega^2 \equiv -1 \,\, (u^2+v^2)$, this is completely multiplicative in the upper variable, so that this is an extension of the character $[r/(u^2+v^2)]$ from $\Z[\zeta_3]$ to $\Z[\zeta_{12}]$.

Similarly as  \cite[Lemma 20.1]{FI} follows from the quadratic reciprocity, the cubic reciprocity law implies that the cubic spin $[z]$ is multiplicative up to the symbol $(z/w)$. The analogous result on $\Z[\sqrt{2}]$ in the work of Milovic is \cite[Proposition 8]{milovic}.

\begin{lemma} \label{twistmultiplemma}
Let  $w=u+iv,z=r+is \in \Z[\zeta_{12}]$ be primary with $w$ primitive. Then
\begin{align*}
[wz]=  [w][z] \bigg(\frac{z}{w} \bigg).
\end{align*}
\end{lemma}
\begin{proof}
First note that since $w$ and $z$ are primary, it follows that all of $u,$ $r,$ $wz,$ and $ur-vs$ are primary, and $3|v$ and $3|s$. If $(u,v) \neq 1$ or $(r,s) \neq 1,$ then the claim is trivial since then both sides vanish. Assume then that $(u,v)=(r,s)=1$. Let $r_0=(r,v)$ be primary, and denote $r=r_0 r_1$, $v = r_0 v_1$, so that $(r_1,v_1)=1$ (since $r$ is primary we have $(3,r_0)=1$ and we may pick a primary representative for $r_0$). By using  $s \equiv u r_1 v_1^{-1} \, (ur_1-v_1 s)$ we get
\begin{align*}
[wz] & =\bigg[\frac{us+vr}{ur-vs}\bigg]=\bigg[\frac{us}{r_0}\bigg]\bigg[\frac{us+vr}{ur_1-v_1 s}\bigg] =\bigg[\frac{u}{r_0}\bigg]\bigg[\frac{s}{r_0}\bigg]\bigg[\frac{u^2r_1v_1^{-1}+vr}{ur_1-v_1 s}\bigg] \\
&=\bigg[\frac{u}{r_0}\bigg]\bigg[\frac{s}{r_0}\bigg]\bigg[\frac{r_1v_1^{-1}}{ur_1-v_1 s}\bigg]\bigg[\frac{u^2+v^2}{ur_1-v_1 s}\bigg] \\
&=\bigg[\frac{u}{r_0}\bigg]^2\bigg[\frac{s}{r_0}\bigg]\bigg[\frac{r_1v_1^{-1}}{ur_1-v_1 s}\bigg]\bigg[\frac{u}{r_0}\bigg]^2\bigg[\frac{u^2+v^2}{ur_1-v_1 s}\bigg] \\
&=\bigg[\frac{u}{r_0}\bigg]^2\bigg[\frac{s}{r_0}\bigg]\bigg[\frac{r_1v_1^{-1}}{ur_1-v_1 s}\bigg]\bigg[\frac{u^2+v^2}{r_0}\bigg]\bigg[\frac{u^2+v^2}{ur_1-v_1 s}\bigg] \\
&=\bigg[\frac{s}{r_0}\bigg]\bigg[\frac{r_1}{ur_1-v_1 s}\bigg] \cdot \bigg[\frac{u}{r_0}\bigg]^2\bigg[\frac{v_1}{ur_1-v_1 s}\bigg]^2 \cdot \bigg[\frac{u^2+v^2}{ur-v s}\bigg].
\end{align*}
We now compute the above three factors separately.

We have
\begin{align*}
&\bigg[\frac{s}{r_0}\bigg]\bigg[\frac{r_1}{ur_1-v_1 s}\bigg] = \bigg[\frac{s}{r_0}\bigg]\bigg[\frac{ur_1-v_1 s}{ r_1}\bigg]  = \bigg[\frac{s}{r_0} \bigg]\bigg[\frac{v_1s}{r_1}\bigg]  =\bigg[\frac{v_1}{r_1}\bigg]  [z]
\end{align*}
by Lemma \ref{cubicreciprocitylemma} since $ur_1-v_1s$ and $r_1$ are both primary.

By a similar argument the second factor is
\begin{align*}
\bigg[\frac{u}{r_0}\bigg]^2\bigg[\frac{v_1}{ur_1-v_1 s}\bigg]^2 &= \bigg[\frac{u}{r_0}\bigg]^2\bigg[\frac{ur_1-v_1 s }{v_1}\bigg]^2\epsilon(v_1,ur_1-v_1 s)^2= \bigg[\frac{u}{r_0}\bigg]^2 \bigg[\frac{ur_1 }{v_1}\bigg]^2 \epsilon(v_1,ur_1-v_1 s)^2\\
&= \bigg[\frac{u}{v}\bigg]^2 \bigg[ \frac{r_1}{v_1}\bigg]^2 \epsilon(v_1,ur_1-v_1 s)^2= [w]^2\bigg[ \frac{r_1}{v_1}\bigg]^2 \epsilon(v_1,ur_1-v_1s)^2\epsilon(u,v)^2 .
\end{align*}

For the third factor, since $u$, $u^2+v^2$, and $ur-vs$ are primary, we have by two applications of Lemma \ref{cubicreciprocitylemma}
\begin{align*}
\bigg[\frac{u^2+v^2}{ur-vs} \bigg] = \bigg[\frac{ur-vs}{u^2+v^2} \bigg] = \bigg[ \frac{u}{u^2+v^2} \bigg] \bigg(\frac{z}{w} \bigg) = \bigg[ \frac{u^2+v^2}{u} \bigg]  \bigg(\frac{z}{w} \bigg)=[w]^2 \bigg(\frac{z}{w} \bigg).
\end{align*}

Combining all we have $[wz]=\mathcal{E} [w][z](\frac{z}{w}) $ for
\begin{align*}
\mathcal{E}&= \bigg[\frac{v_1}{r_1}\bigg]  \cdot \bigg[ \frac{r_1}{v_1}\bigg]^2 \epsilon(v_1,ur_1-v_1s)^2 \epsilon(u,v)^2  \\
&=\epsilon(v_1,r_1)   \epsilon(v_1,ur_1-v_1s)^2 \epsilon(v,u)  
\end{align*}
Let $v_1=\pm \zeta_3^k(1-\zeta_3)^\ell \lambda$ where $\lambda$ is primary. Then, since $9|v_1s$, we see from the supplementary laws in Lemma \ref{cubicreciprocitylemma} (note also that $\epsilon(\lambda,ur_1-v_1s)=1= \epsilon(\lambda,ur_1)$ since $ur_1-v_1s$ and $ur_1$ are primary)
\begin{align*}
 \epsilon(v_1,ur_1-v_1s) &= \epsilon(\zeta_3,ur_1-v_1s)^k \epsilon(1-\zeta_3,ur_1-v_1s)^\ell \epsilon(\lambda,ur_1-v_1s) \\
 &=  \epsilon(\zeta_3,ur_1)^k \epsilon(1-\zeta_3,ur_1)^\ell \epsilon(\lambda,ur_1) = \epsilon(v_1,ur_1) \\
 & = \epsilon(v_1,u) \epsilon(v_1,r) =\epsilon(r_0,u)^2\epsilon(v,u) \epsilon(v_1,r) = \epsilon(v,u) \epsilon(v_1,r)
\end{align*}
since $r_0$ and $u$ are primary. Hence, we get $\mathcal{E}=\epsilon(v_1,r_1) \epsilon(v,u)^2 \epsilon(v_1,r_1)^2 \epsilon(v,u)=1.$
\end{proof}

\begin{remark}
It is perhaps surprising that our Lemma 5 is much simpler than the corresponding results in the quadratic case \cite[Lemma 20.1]{FI} and \cite[Proposition 8]{milovic}, where the equalities hold only up to a simple factor. The main reason for this seems to be that $-1$ is always a cube so that $[-1/a] = 1$ for all $a \in \Z[\zeta_3]$, a fact that we used in the proof.
\end{remark}

We will abbreviate the norm maps as follows. For any $a=a^{(1)}+a^{(2)} \zeta_3 \in \Z[\zeta_3]$ and $\zeta=r+is \in \Z[\zeta_{12}]$, $r,s \in \Z[\zeta_3]$ we set
\begin{align*}
&N_{3}(a):= N_{\Q(\zeta_3)}(a) = (a^{(1)})^2 - a^{(1)} a^{(2)} +  (a^{(2)})^2 = |a|^2, \\
 N_{12/3} (\zeta)&:= N_{\Q(\zeta_{12})/\Q(\zeta_{3})}(\zeta)= r^2+s^2, \quad \text{and} \quad N_{12} := N_{\Q(\zeta_{12})} = N_{3} \circ N_{12/3}.
\end{align*}

The symbol $(z/w)$ in Lemma \ref{twistmultiplemma} is completely multiplicative in the upper variable. Similar to \cite[Sections 19 and 21]{FI}, to handle Type II sums we need a multiplier rule in the lower variable also, which is given by the following.
\begin{lemma}\label{symbolmultiplemma} Let $w_1,w_2 \in \Z[\zeta_{12}]$ be primary primitive and let $\zeta=r+is \in \Z[\zeta_{12}]$ be primary. Set 
\begin{align*}
q_j= N_{12/3}(w_j), \quad e:=(w_1,\sigma(w_2^2)),\quad \text{and} \quad d:= N_{12/3}(e),
\end{align*}
where $\sigma$ is the conjugation $\sigma(a+ib)=a-ib$. Then for some root $\omega^2+1 \equiv 0 \,\,(q_1q_2^2)$ we have
\begin{align*}
\bigg(\frac{\zeta}{w_1} \bigg) \bigg(\frac{\zeta}{w_2} \bigg)^2 &= \bigg[\frac{r-\omega s}{d} \bigg] \bigg[\frac{r+\omega s}{q_1q_2^2/d} \bigg].
\end{align*}
\end{lemma}
\begin{proof}
We can write $\zeta= a z$, where $a \in \Z[\zeta_{3}]$ and $z$ is primary primitive. Since both sides are completely multiplicative, it suffices to prove the claim separately for  $a \in \Z[\zeta_3]$ and for $z\in \Z[\zeta_{12}]$ primary primitive.

For any $a \in \Z[\zeta_3]$ we have by definition
\begin{align*}
\bigg(\frac{a}{w_1} \bigg) \bigg(\frac{a}{w_2} \bigg)^2 = \bigg[\frac{a}{q_1}\bigg] \bigg[\frac{a}{q_2^2}\bigg] = \bigg[\frac{a}{d}\bigg]\bigg[\frac{a}{q_1q_2^2/d}\bigg].
\end{align*}

For $z=r+is$ primary primitive we get from Lemma \ref{twistmultiplemma} a reciprocity law $(z/w)=(w/z)$ for any primary primitive $w$. Note that by definition $w_1w_2^2/d$ is primary primitive. Therefore, we get by reciprocity (note that $r^2+s^2$ and $d$ are primary)
\begin{align*}
\bigg(\frac{z}{w_1} \bigg) \bigg(\frac{z}{w_2} \bigg)^2 &=\bigg(\frac{w_1}{z} \bigg) \bigg(\frac{w_2^2}{z} \bigg)  =\bigg(\frac{d}{z} \bigg)\bigg(\frac{w_1 w_2^2/d}{z} \bigg) = \bigg[\frac{d}{N_{12/3} (z)} \bigg]\bigg(\frac{w_1 w_2^2/d}{z} \bigg)\\
& = \bigg[\frac{r^2+s^2}{d} \bigg] \bigg[\frac{r+\omega s}{q_1q_2^2 /d^2} \bigg] = \bigg[\frac{r-\omega s}{d} \bigg] \bigg[\frac{r+\omega s}{q_1q_2^2/d} \bigg],
\end{align*}
since $r^2+s^2 \equiv (r+\omega s) (r-\omega s) \,\, (d)$.
\end{proof}

\begin{remark} In \cite[Section 19]{FI} we have for primary $z=r+is,w=u+iv \in \Z[i]$ with $w$ primitive
\begin{align*}
\bigg( \frac{z}{w}\bigg)_2:= \bigg(\frac{r+\omega s}{u^2+v^2}\bigg)_2 = \bigg(\frac{ur-vs}{u^2+v^2}\bigg)_2 =\bigg( \frac{\text{Re}\,wz}{|w|}\bigg)_2.
\end{align*}
In our case the middle equality does not hold but we have for $z=r+is,w=u+iv \in \Z[\zeta_3]$ primary
\begin{align*}
\bigg(\frac{ur-vs}{u^2+v^2}\bigg) =\bigg(\frac{u}{u^2+v^2}\bigg)\bigg(\frac{r+\omega s}{u^2+v^2}\bigg) =[w]^2\bigg( \frac{z}{w}\bigg) 
\end{align*}
by reciprocity if $w$ is primary primitive. Lack of this alternative representation does not hinder us in any way.
\end{remark}
\section{The twelfth cyclotomic extension} \label{12section}
So far we have defined $[z]$ only for primary $z \in \Z[\zeta_{12}]$. In this section we extend the definition to ideals of $\Z[\zeta_{12}]$.

The unit group of $\Z[\zeta_{12}]$ is generated by $\zeta_{12}$ and the fundamental unit (cf. \cite[Chapter 7.4]{lemm})
\begin{align*}
\epsilon_0 := \frac{1+\sqrt{3}}{1-i} = 1+ \zeta_3 - i\zeta_3.
\end{align*}
For every $z \in \Z[\zeta_{12}]$ coprime to 3 there exists a unit $\mu$ such that $\mu z \equiv 1 \, \, (3)$ (cf. \cite[Exercise 7.4]{lemm}). The next lemma shows that the subgroup of primary units is $\{\pm (i \epsilon_0^6)^k: \, k \in \Z\}$.

\begin{lemma} \label{unitmod3lemma} We have $-i \epsilon_0^6 \equiv 1 \,\, (3)$. Furthermore, $k=6$ is the smallest positive exponent such that $\epsilon_0^k \equiv \zeta_{12}^{\ell} \,\, (3)$ for some integer $\ell$, so that as a set
\begin{align*}
(\Z[\zeta_{12}]/3\Z[\zeta_{12}])^{\times} = \{ \zeta_{12}^{\ell} \epsilon_0^{k}, \,\, k \in \{0,1, \dots, 5\}\}.
\end{align*}
 Also, we have
\begin{align*}
[\pm i\epsilon_0^6] =1.
\end{align*}
\end{lemma}
\begin{proof}
By direct computation we see that
\begin{align*}
i\epsilon_0^6 = 26 + i(15+30\zeta_3) \equiv 2 \,\, (3),
\end{align*}
and that for $1\leq k \leq 5$ and for all $\ell$ we have $\epsilon_0^k \zeta_{12}^\ell \not\equiv \pm 1 \,\, (3)$ (see Table \ref{tablepower}). Note that to check this it suffices to verify that for all $1\leq k \leq 5$ and $0 \leq \ell \leq 2$ we have $\epsilon_0^k \zeta_{12}^\ell \not\equiv \pm 1, \pm i \,\, (3)$.

Since every number coprime to 3 is congruent to some unit modulo 3, this implies the claimed structure for $(\Z[\zeta_{12}]/3\Z[\zeta_{12}])^{\times}$. By the definition of $[z]$ we get
\begin{align*}
[i\epsilon_0^6] = \bigg[ \frac{15+30\zeta_3}{26} \bigg] =\bigg[ \frac{15}{26} \bigg]\bigg[ \frac{1+2\zeta_3}{26} \bigg] =\bigg[ \frac{\zeta_3}{26}\bigg]\bigg[ \frac{1-\zeta_3}{26}\bigg] = \zeta_3^{(1+26)/3} \zeta_3^{(-1-26)/3}=1
\end{align*}
by the supplementary laws in Lemma \ref{cubicreciprocitylemma}, and since $[m/n]=1$ for all $m,n \in \Z$ with $n \neq 0$.
\end{proof}

\begin{center}
\begin{table}[t] \label{tablepower}
\begin{tabular}{ c|c| c| c }
\backslashbox{$k$}{$\ell$} & 0 & 1 & 2  \\
\hline
1 & $\frac{1}{2} + \frac{i}{2} + \frac{\sqrt{3}}{2} + \frac{i \sqrt{3}}{2}$& $\frac{1}{2} + i +\frac{i\sqrt{3}}{2}$ & $\frac{-1}{2} + i +\frac{i \sqrt{3}}{2}$ \\
\hline
2 & $2i + i \sqrt{3}$ & $-1 + \frac{3i}{2} - \frac{\sqrt{3}}{2} - i \sqrt{3}$ & $\frac{-3}{2} + i - \sqrt{3} + \frac{i \sqrt{3}}{2} $\\
\hline
3 & $ \frac{-5}{2} +\frac{5i}{2} -\frac{3 \sqrt{3}}{2} + \frac{3 i\sqrt{3}}{2} $ & $\frac{-7}{2}+ i -2 \sqrt{3} + \frac{i \sqrt{3}}{2}$ & $\frac{-7}{2} -i - 2 \sqrt{3} - \frac{i \sqrt{3}}{2}$\\
\hline
4 & $-7-4 \sqrt{3}$ & $-6 -\frac{7i}{2} - \frac{7 \sqrt{3}}{2} -2 i \sqrt{3}$  & $\frac{-7}{2} - 6i -2 \sqrt{3} -\frac{7i\sqrt{3}}{2}$ \\
\hline
5 & $\frac{-19}{2} -\frac{19i}{2} -\frac{11\sqrt{3}}{2}-\frac{11i\sqrt{3}}{2}$ & $\frac{-7}{2} - 13i-2 \sqrt{3} - \frac{15i\sqrt{3}}{2}$ & $\frac{7}{2}-13i + 2 \sqrt{3} - \frac{15i\sqrt{3}}{2}$
\end{tabular}
\caption{Values of $\epsilon_0^k \zeta_{12}^\ell$. Note that for $z \in \Z[\zeta_{12}]$ if $z \equiv \pm 1, \pm i \, (3)$, then both of the coefficients of $\sqrt{3}/2$ and $i\sqrt{3}/2$ are divisible by 3. The only entry satisfying this is $\epsilon_0^3 = \frac{-5}{2} +\frac{5i}{2} -\frac{3 \sqrt{3}}{2} + \frac{3 i\sqrt{3}}{2} ,$ but we have $\epsilon_0^3 = 5-2i -3(1+i)\zeta_3 \equiv 2-2i \not \equiv \pm 1, \pm i \, (3)$, so that none of the values in the table are $\equiv \pm 1, \pm i \, (3)$.}
\end{table}
\end{center}

As a corollary we see that $[z]$ does not depend on which primary associate we choose.
\begin{lemma}
If $z$ and $z'$ are primary associates, then $[z]=[z']$.
\end{lemma}
\begin{proof}
There is some unit $\mu$ such that $z=\mu z'$. Since $z$ and $z'$ are primary, also $\mu$ must be primary. By Lemma \ref{unitmod3lemma} we see that $\mu = \pm(i\epsilon_0^6)^k$ for some $k \in \Z$. Hence, by Lemma \ref{twistmultiplemma}
\begin{align*}
[z] = [\mu z'] = [\mu][z'] \bigg( \frac{z'}{\mu} \bigg).
\end{align*}
By definition
\begin{align*}
\bigg( \frac{z'}{\mu} \bigg) = \bigg[ \frac{r'+\omega s'}{N_{12/3}(\mu)}\bigg] = 1
\end{align*}
since $N_{12/3}(\mu)$ is a unit in $\Z[\zeta_3]$, and similarly we see that
\begin{align*}
[\mu] = [\pm(i\epsilon_0^6)^k] = [i\epsilon_0^6]^k =1
\end{align*}
by using the last part of Lemma \ref{unitmod3lemma}.
\end{proof}
Since $\Z[\zeta_{12}]$ is a principal ideal domain, by the above lemma the following definition is appropriate.
\begin{definition} \label{cubicdefinition} For any ideal $\a$ of $\Z[\zeta_{12}]$, we define
\begin{align*}
[\a] := [z] = \bigg[\frac{s}{r}\bigg]
\end{align*}
if $\a$ is generated by $z=r+is$ and $z$ is primary.
\end{definition}
\begin{remark}
We  also define $[z]$ for non-primary $z=r+is$ by $[z]:=[s/r]$. Then we have $[z] = \nu(z)[(z)]$ where $\nu(z)$ depends only on the residue class $z \, (3)$. We will not need it in the following but it might be interesting to give a simple closed formula for $\nu(z)$.
\end{remark}
\begin{remark}
It is natural that we need the reciprocity laws to prove that the spin is well-defined on ideals. After all, the cubic reciprocity can be restated as a transformation rule for $[z]$ under multiplication of $z$ by a root of unity. This is because for $z=r+is$
\begin{align*}
[iz] =[-s+ir] =\bigg[\frac{r}{s} \bigg]& = \epsilon(r,s) [z], \quad  \text{and}  \quad [\zeta_3 z] = \bigg[ \frac{\zeta_3 s}{\zeta_3 r}\bigg] = \bigg[ \frac{\zeta_3 }{r}\bigg]\bigg[ \frac{s}{ r}\bigg] = \epsilon(\zeta_3,r) [z],
\end{align*}
which by considering $z_1=r+i(1-\zeta_3)$ and $z_2=r+i3$ covers also the supplementary laws.
\end{remark}

\section{Sieve argument} \label{sievesection}
In this section we prove Theorem \ref{maintheorem}. We apply a sieve argument in $\Z[\zeta_{12}]$ to decompose our sum into Type I and Type II sums. The argument is essentially same as in Harman's sieve method \cite{harman}. We could use \cite[Proposition 5.2]{fimr} directly but we give our on proof based on Buchstab's identity since it does not take much effort to inclue it here.

For two functions $f$ and $g$ with $g \geq 0$, it is convenient for us to denote  $f(N) \pprec g(N)$ if $f(N) \ll_\epsilon N^{\epsilon}g(N)$. A typical bound we use is $\tau_k(n) \pprec 1$, where $\tau_k$ is the $k$-fold divisor function. For multivariable functions such as sums over two variables we write
\begin{align*}
\sum_{\substack{m \sim M \\ n \sim N}} f(m,n) \pprec \sum_{\substack{m \sim M \\ n \sim N}} g(m,n)
\end{align*}
to mean
\begin{align*}
\sum_{\substack{m \sim M \\ n \sim N}} f(m,n) \ll_\epsilon (M+N)^{\epsilon} \sum_{\substack{m \sim M \\ n \sim N}} g(m,n).
\end{align*}
 We say that an arithmetic function $f$ is divisor bounded if $|f(n)| \ll \tau_k(n)$ for some $k$.

For the sieve we require the following arithmetic information, which is proved in Sections \ref{typeisection} and \ref{typeiisection}.

\begin{prop} \emph{\textbf{(Type I sums).}}\label{typeiprop} Let $\alpha_\d$ be divisor bounded. Then
\begin{align*}
\sum_{N_{12}(\d) \leq D} \alpha_\d \sum_{N_{12}(\n)\sim x/N_{12}(\d)} [\d \n] \pprec x^{11/12} D^{13/12}
\end{align*}
\end{prop}

\begin{prop} \emph{\textbf{(Type II sums).}}\label{typeiiprop} Let $\alpha_\m$ and $\beta_\n$ be divisor bounded coefficients. Then
\begin{align*}
\sum_{N_{12}(\m) \sim M} \sum_{N_{12}(\n) \sim N} \alpha_\m \beta_\n [\m \n] \pprec MN^{9/10} + M^{9/10} N.
\end{align*}
\end{prop}

\begin{remark} Note that $[\a]=0$ if $\a$ is not primitive. Thus, we may assume that the coefficients $\alpha$ and $\beta$ in the above are supported on primitive ideals.
\end{remark}

\begin{remark} Our Type I information is very weak but this is compensated by the fact that the Type II bound is non-trivial as soon as $M \gg x^\eta$ or $N \gg x^\eta$.
\end{remark}
Define
\begin{align*}
P_{12}(Y):= \prod_{\substack{N_{12} (\p) < Y }} \p.
\end{align*}
Note that the norm map induces a partial ordering on the set of ideals, and that for every prime ideal there are at most four prime ideals of the same norm. For any ideal $\d \subseteq \Z[\zeta_{12}]$ we set
\begin{align*}
S(\A_\d, Y) := \sum_{\substack{\n \\ N_{12}(\d \n) \sim x}} 1_{(\n,P_{12}(Y))=1} [\d \n].
\end{align*}
\emph{Proof of Theorem \ref{maintheorem}.}
Let $Z=x^\gamma$ for some $\gamma \in (0,1/2)$ which we will optimize later on. By Buchstab's identity
\begin{align*}
S(\A, 2 \sqrt{x}) = S(\A, Z)  - \sum_{\substack{Z \leq N_{12}(\p)  < 2 \sqrt{x}}} S(\A_\p, N_{12}(\p)) + O(E) =: S_{1}(\A) - S_{2}(\A) + O(E),
\end{align*}
where the error term $E$ consists of that part in $S_{2}(\A)$ where the implicit variable $\n$ is divisible by another prime ideal of same norm $N_{12}(\p)$ (we could also handle this part by fixing a complete ordering for prime ideals $\p$ but then we would later have to remove cross-conditions coming from this). We have trivially
\begin{align*}
E \leq \sum_{\substack{Z \leq N_{12}(\p)  < 2 \sqrt{x}}} \sum_{\substack{\n \\ N_{12}(\n) \sim x}} 1_{N_{12}(\p)^2 | N_{12}(\n)} \pprec \sum_{Z \leq k < 2\sqrt{x}} \sum_{n \sim x} 1_{k^2 | n} \ll x Z^{-1}.
\end{align*}

For the second sum we have (writing $\n= \p_1  \cdots \p_k$)
\begin{align} \nonumber
S_{2}(\A) &= \sum_{\substack{Z \leq N_{12}(\p)  < 2 \sqrt{x} }} \sum_{\substack{\n \\ N_{12}(\p \n) \sim x}} 1_{(\n,P_{12}(N_{12}(\p)))=1} [\p \n] \\ \nonumber
&= \sum_{k \ll 1} \sum_{\substack{Z \leq N_{12}(\p)  < 2 \sqrt{x}}} \, \sum_{\substack{N_{12}(\p) < N_{12}(\p_1 )< \cdots < N_{12}(\p_k)  \\ N_{12}(\p \p_1\cdots \p_k) \sim x}} [\p \p_1\cdots \p_k ] + O(\tilde{E})\\  \label{bound1}
& \pprec xZ^{-1/10} + xZ^{-1}
\end{align}
by Proposition \ref{typeiiprop} once we remove the cross-conditions $N_{12}(\p) < N_{12}(\p_1)$ and $N_{12}(\p \p_1\cdots \p_k) \sim x$ by Perron's formula (cf. for instance \cite[Chapter 3.2]{harman}, this works essentially the same in our situation since we apply it to the real quantities $N_{12}(\p_j)$). Here the error term $\tilde{E}$ consists of the part where $\p\p_1\cdots \p_k$ is divisible by a square, so that we have $\tilde{E} \ll x Z^{-1}$by a similar argument as with $E$ above.

For the first sum $S_1(\A)$ we use the M\"obius function to expand the condition $(\n,P_{12}(Z))=1$ to get
\begin{align*}
S_{1}(\A) &= \sum_{\d | P_{12}(Z)} \mu(\d) \sum_{\substack{\n \\ N_{12}(\d \n) \sim x}}  [\d\n] \\
&= \sum_{\substack{\d | P_{12}(Z)\\ N_{12}(\d) < Z}} \mu(\d) \sum_{\substack{\n \\ N_{12}(\d\n) \sim x}}  [\d \n]+ \sum_{\substack{\d | P_{12}(Z) \\ N_{12}(\d) \geq Z}} \mu(\d) \sum_{\substack{\n \\ N_{12}(\d \n) \sim x}}  [\d \n] =: S_{11}(\A) + S_{12}(\A)
\end{align*}
For the first sum we apply Proposition \ref{typeiprop} to get
\begin{align} \label{bound2}
S_{11}=\sum_{\substack{\d | P_{12}(Z)\\ N_{12}(\d) < Z}} \mu(\d) \sum_{\substack{\n \\ N_{12}(\d\n) \sim x}}  [\d \n]\pprec x^{11/12} Z^{13/12}.
\end{align}
For the second sum we write $\d = \p_1\cdots \p_k$ for $N_{12}(\p_1) \leq \cdots \leq N_{12}(\p_k) < Z$. Since there are at most four prime ideals of the same norm and $\d$ is square free, by the greedy algorithm there is a unique $\ell \leq k$ such that $\d=\d_1\d_2$ with
\begin{align*}
\d_1 =& \p_1\cdots \p_\ell, \quad \quad \d_2 = \p_{\ell+1} \cdots \p_k, \\
N_{12}(\d_1) \in [Z,Z^5], &\quad N_{12}(\d_1') < Z, \quad \text{\and} \quad N_{12}(\p_\ell) < N_{12}(\p_{\ell+1}),
\end{align*}
where $\d_1' := \p_1 \cdots \p_{\ell-j}$ if $j$ is the largest number such that $N_{12}(\p_{\ell-j+1}) =  N_{12}(\p_{\ell})$ (that is, we apply the greedy algorithm for groups of at most four prime ideals of the same norm). Hence, the second sum $S_{12}(\A)$ can be partitioned as
\begin{align*}
\sum_{k \ll \log x}(-1)^k \sum_{\ell \leq k} \sum_{\substack{N_{12}(\d_1) \in [Z,Z^5] \\ N_{12}(\d_1') < Z \\
\d_1=\p_1 \cdots \p_\ell  \, \text{square free}\\
N_{12}(\p_1) \leq \cdots \leq N_{12}(\p_\ell) < Z}} \sum_{\substack{\d_2 = \p_{\ell+1} \cdots \p_k \, \text{square free}\\ 
N_{12}(\p_{\ell}) < N_{12}(\p_{\ell+1}) \leq \cdots \leq N_{12}(\p_k) < Z}} \sum_{\substack{\n \\ N_{12}(\d_1)N_{12}(\d_2 \n) \sim x}}  [\d_1 \d_2 \n] .
\end{align*}
The cross-conditions $N_{12}(\p_{\ell}) < N_{12}(\p_{\ell+1})$ and $N_{12}(\d_1)N_{12}(\d_2 \n) \sim x$ can now be removed by Perron's formula, so that by Proposition \ref{typeiiprop} we get
\begin{align} \label{bound3}
\sum_{\substack{\d | P(Z) \\ N_{12}(\d) \geq Z}} \mu(\d) \sum_{\substack{\n \\ N_{12}(\d \n) \sim x}}  [\d \n] \pprec \frac{x}{Z}Z^{9/10} + Z^5\bigg(\frac{x}{Z^5}\bigg)^{9/10}.
\end{align}
 Combining the bounds (\ref{bound1}), (\ref{bound2}), and (\ref{bound3}), and choosing $Z:=x^{5/71}$ (note that then $Z^5 < x^{1/2}$) to optimize we get
\begin{align*}
S(\A, 2 \sqrt{x}) \pprec x Z^{-1/10} + x^{11/12} Z^{13/12} \ll_\epsilon x^{1-1/142+\epsilon} \ll x^{1-1/143}.
\end{align*} \qed
\section{Fixing primary generators} \label{fixingsection}
For the proofs of our arithmetic information (Propositions \ref{typeiprop} and \ref{typeiiprop}) we need to fix primary generators of ideals of $\Z[\zeta_{12}]$ in a consistent manner, and in such a way that the resulting conditions do not cause problems later on. Luckily fixing an embedding of $\Z[\zeta_{12}]$ in $\C$ along with Lemma \ref{unitmod3lemma} allows us to do just this. We choose the embedding which maps $\zeta_{12}$ to $e^{2\pi i /12} \in \C$. For any $z=r+is \in \Z[\zeta_{12}]$, $r,s \in \Z[\zeta_3]$ we let $|z|:= |r+is|$ denote the norm of the complex number $r+is \in \mathbb{C}$. Note that then
\begin{align*}
|\epsilon_0| = \bigg| \frac{1+\sqrt{3}}{1-i}\bigg| = (1+\sqrt{3})/\sqrt{2} >1.
\end{align*}
\begin{lemma} \label{fixinglemma} For every ideal $\a$ coprime to 3 there exists a unique generator $z=r+is \in \Z[\zeta_{12}]$ of $\a$ such that $z \equiv 1 \,\, (3)$ and 
\begin{align} \label{normconstraint}
N_{12}(z)^{1/4} \leq |z| < N_{12}(z)^{1/4} |\epsilon_0|^6.
\end{align}
Furthermore, for such a $z=r+is$ we have $|r|, |s| \, \ll |z|$.
\end{lemma}
\begin{proof}
If $z_0 \equiv 1 \,\, (3)$ is a  generator of $\a$, then by Lemma \ref{unitmod3lemma} the associates of $z_0$ which are $\equiv 1 \, (3)$  are precisely $(-i\epsilon_0^6)^{k} z_0$ with $k \in \Z$. Clearly there is a unique $k$ such that $z=(-i\epsilon_0^6)^{k} z_0$ satisfies (\ref{normconstraint}). From (\ref{normconstraint}) it follows that
\begin{align*}
|r^2+s^2|^{1/2} \, \asymp \, |r+is|,
\end{align*}
so that $|r-is| \, \asymp |r+is|$ which implies $|r|, |s| \, \ll |z|$.
\end{proof}
In the summations we will denote this condition by $\sumw$, so that we may write, for example,
\begin{align}
\sumw_{N_{12}(z) \sim N}  [z] \,= \sum_{\substack{N_{12}(z) \sim N \\
z \equiv 1 \,\, (3) \\ N_{12}(z)^{1/4} \leq |z| < N_{12}(z)^{1/4} |\epsilon_0|^6}} [z] \, = \sum_{N_{12}(\n) \sim N} [\n].
\end{align}

\section{Truncated Poisson summation formula on $\Z[\zeta_{3}]$} \label{poissonsection}
In the proofs of Propositions \ref{typeiprop} and \ref{typeiiprop} we will need a version of the Poisson summation formula on $\Z[\zeta_3]$. For the lemma we fix an embedding identifying $\zeta_3$ with $e^{2 \pi i /3} \in \C$, so that any element of $\Z[\zeta_3]$ is viewed as a complex number. For $z \in \C$ denote
\[z=z^{(1)} + \zeta_3 z^{(2)}\]
 with $z^{(1)},z^{(2)} \in \R$. For $q \in \Z[\zeta_{3}] \setminus \{0\}$, $\beta \in \C$, and $h_1,h_2 \in \Z$ we define (denoting $e_q(x):= e^{2\pi i  x /q}$)
\begin{align*}
\psi_q^{(h_1,h_2)}(\beta) := e_{N_3(q)}(h_1(\beta \bar{q})^{(1)} + h_2(\beta \bar{q})^{(2)}),
\end{align*}
so that $\psi_q^{(h_1,h_2)}$ is an additive character of $\Z[\zeta_{3}]/q\Z[\zeta_{3}]$ (here $\bar{q}$ denotes the complex conjugate).

\begin{lemma} \label{poissonlemma} \emph{\textbf{(Truncated Poisson summation formula on $\Z[\zeta_{3}]$).}} Fix $\beta,q \in \Z[\zeta_{3}]$ and $x_0,y_0 \in \R$. For $K > 1$ let 
\begin{align*}
G_K(x,y):= G\bigg(\frac{x-x_0}{K}, \frac{y-y_0}{K}\bigg)
\end{align*}
for some fixed compactly supported $C^{\infty}$-smooth function $G: \R^2 \to \C$. We define $G_K:\C \to \C$ by setting $G_K(\alpha):= G_K(\alpha^{(1)},\alpha^{(2)})$. Then for any $C, \epsilon > 0$ with $H:= K^{\epsilon/2}N_3(q)^{1/2}/K$ we have
\begin{align*}
\sum_{\substack{\alpha \in \Z[\zeta_{3}] \\ \alpha \equiv \beta \,\, (q)}}G_K(\alpha) =\frac{1}{N_3(q)}\sum_{\substack{\alpha \in \Z[\zeta_{3}] }}G_K(\alpha)+\frac{K^{\epsilon}}{H^2} \sum_{\substack{|h_1|,|h_2|\leq H \\ (h_1,h_2) \neq (0,0) } } c_{h_1,h_2}  \psi_q^{(h_1,h_2)}(-\beta) + O_{G,C,\epsilon}(K^{-C})
\end{align*}
for some bounded coefficients $c_{h_1,h_2}=c_{h_1,h_2,G,K,q}$ satisfying $|c_{h_1,h_2}|\ll_G 1$.
\end{lemma}
\begin{proof} Recall that for $\alpha,q \in \Z[\zeta_3]$ we have
\begin{align*}
\alpha q =  \alpha^{(1)} q^{(1)} - \alpha^{(2)} q^{(2)} + \zeta_3 (\alpha^{(2)} q^{(1)} + \alpha^{(1)} q^{(2)} - \alpha^{(2)} q^{(2)}).
\end{align*}
Substituting $\alpha \mapsto \alpha q+\beta$, we get by two applications of the usual Poisson summation formula (denoting $x:=x_1+\zeta_3 x_2 \in \C$)
\begin{align*}
&\sum_{\substack{\alpha \in \Z[\zeta_{3}] \\ \alpha \equiv \beta \,\, (q)}}G_K(\alpha) = \sum_{\alpha \in \Z[\zeta_{3}]}  G_K(\alpha q +\beta) = \sum_{\alpha \in \Z[\zeta_{3}]} G_K((\alpha q)^{(1)} + \beta^{(1)},(\alpha q)^{(2)} + \beta^{(2)})\\ 
&=  \sum_{\alpha^{(1)}, \alpha^{(2)} \in \Z} G_K(\alpha^{(1)} q^{(1)} - \alpha^{(2)} q^{(2)} + \beta^{(1)},\alpha^{(2)} q^{(1)} + \alpha^{(1)} q^{(2)} - \alpha^{(2)} q^{(2)} + \beta^{(2)})\\\
& = \sum_{h_1,h_2} \iint G_K(x_1 q^{(1)} - x_2 q^{(2)} + \beta^{(1)},x_2 q^{(1)} + x_1 q^{(2)} - x_2 q^{(2)} + \beta^{(2)}) e(h_1x_1+h_2x_2) d x_1 d x_2\\
&= \sum_{h_1,h_2} \iint G_K((xq)^{(1)}+\beta^{(1)},(xq)^{(2)}+\beta^{(2)}) e(h_1x_1+h_2x_2) d x_1 d x_2 \\
&= \sum_{h_1,h_2} \iint G_K(xq +\beta) e(h_1x_1+h_2x_2) d x_1 d x_2,
\end{align*}
where $xq$ is computed as a multiplication of two complex numbers, so that 
\[
xq = (xq)^{(1)} + \zeta_3 (xq)^{(2)}.
\]
We make the change of variables 
\begin{align*}
x_j \mapsto x_j K/N_3(q)^{1/2} - (\beta \bar{q})^{(j)}/N_3(q) \quad  \text{for}\quad j\in \{1,2\}
\end{align*}
so that $xq+\beta$ is mapped to
\begin{align*}
&\bigg(x_1 K/N_3(q)^{1/2}  + \zeta_3 x_2 K/N_3(q)^{1/2}  - \beta \bar{q} / N_3(q)\bigg) \cdot q + \beta =\frac{K}{N_3(q)^{1/2}} xq.
\end{align*}
We get
\begin{align*}
\sum_{\substack{\alpha \in \Z[\zeta_{3}] \\ \alpha \equiv \beta \,\, (q)}}G_K(\alpha) = \frac{K^2}{N_3(q)} \sum_{h_1,h_2} c_{h_1,h_2}  \psi_q^{(h_1,h_2)}(-\beta),
\end{align*}
where
\begin{align*}
&c_{h_1,h_2} := \iint G_K \bigg(\frac{K}{N_3(q)^{1/2}} xq\bigg)  e\bigg(\frac{K(h_1x_1+h_2x_2)}{N_3(q)^{1/2}}\bigg) d x_1 d x_2  \\
& = \iint G\bigg(\frac{(xq)^{(1)}}{N_3(q)^{1/2}}-\frac{x_0}{K},\frac{(xq)^{(2)}}{N_3(q)^{1/2}}-\frac{y_0}{K}\bigg) e\bigg(\frac{K(h_1x_1+h_2x_2)}{N_3(q)^{1/2}}\bigg) d x_1 d x_2 \\
&=\psi_q^{(h_1,h_2)}(x_0 + \zeta_3 y_0)  \iint G\bigg(\frac{(xq)^{(1)}}{N_3(q)^{1/2}},\frac{(xq)^{(2)}}{N_3(q)^{1/2}}\bigg) e\bigg(\frac{K(h_1x_1+h_2x_2)}{N_3(q)^{1/2}}\bigg) d x_1 d x_2
\end{align*}
by making the translation (denoting $z_0:=(x_0 + \zeta_3 y_0)$)
\begin{align*}
x \mapsto x+ \frac{1}{K N_3(q)^{1/2}}z_0 \bar{q},
\end{align*}
that is,
\begin{align*}
x_j \mapsto x_j + \frac{1}{K N_3(q)^{1/2}} (z_0 \bar{q})^{(j)}.
\end{align*}

For all $h_1,h_2$ we have the trivial estimate $|c_{h_1,h_2}| \leq c_{0,0} \ll_G 1$ (note that $x \mapsto xq/N^{1/2}_3(q)$ is a rotation in $\C$ so that $c_{0,0}$ is independent of $q$). As usual, $h_1=h_2=0$ gives us the main term, since by another double application of Poisson summation
 \begin{align*}
 \frac{K^2}{N_3(q)} c_{0,0} =  \frac{K^2}{N_3(q)}  \iint G\bigg(\frac{(xq)^{(1)}}{N_3(q)^{1/2}},\frac{(xq)^{(2)}}{N_3(q)^{1/2}}\bigg) d x_1 d x_2 = \frac{K^2}{N_3(q)}  \iint G(x_1,x_2) d x_1 d x_2 \\
 =  \frac{1}{N_3(q)}  \iint G_K(x_1,x_2) d x_1 d x_2 =  \frac{1}{N_3(q)}\sum_{\substack{\alpha \in \Z[\zeta_{3}] }}G_K(\alpha) + O_{C}(K^{-C}).
 \end{align*}

For $|h_1| >H$ or $|h_2| > H$ we can iterate integration by parts to show that the contribution from this part is $\ll_{C,\epsilon} K^{-C}.$ 
\end{proof}

\section{Type I sums} \label{typeisection}
Using the notation of Section \ref{fixingsection}, for any primary $w \in \Z[\zeta_{12}]$ define
\begin{align*}
\K_w(N):= \sumw_{N_{12}(z) \sim N} [wz]
\end{align*}
In this section we show (analogously to \cite[Proposition 22.1]{FI}) the following proposition, which implies Proposition \ref{typeiprop}.
\begin{prop} \label{linearsumprop}We have
\begin{align*}
\K_w(N)  \pprec N^{11/12} N_{12}(w)^{1/6}
\end{align*}
\end{prop}
For the proof we need a generalization of the P\'olya-Vinogradov estimate for short character sums. Unfortunately on $\Z[\zeta_{3}]$ we do not quite have the usual square-root bound since the estimate relies on the error term in counting lattice points; on a number field of degree $d$ and a primitive non-principal character $\chi$ modulo $\mathfrak{q}$ Landau's generalization gives 
\begin{align*}
\sum_{N(\mathfrak{a}) \leq x} \chi(\mathfrak{a}) \ll (N \mathfrak{q})^{1/(d+1)} (\log N (\mathfrak{q}))^d x^{(d-1)/(d+1)}.
\end{align*}
However, for a smoothed version we have the P\'olya-Vinogradov estimate in the usual form. Also, since we require the bound not only for short sums near 0 but in more general small sets, the smoothed version is more convenient for us. Unfortunately in our application we need to transition from a smoothed version to a sharp cut-off, which causes us to lose a power of $x$ compared to the results in \cite{FI}.

\begin{lemma} \label{pvlemma} Fix $\beta,q \in \Z[\zeta_{3}]$ and $x_0,y_0 \in \R$. For $K \gg 1$ let 
\begin{align*}
G_K(x,y):= G\bigg(\frac{x-x_0}{K}, \frac{y-y_0}{K}\bigg)
\end{align*}
for some fixed compactly supported $C^{\infty}$-smooth function $G$. For $\alpha \in \Z[\zeta_{3}]$, let $G_K(\alpha):= G_K(\alpha^{(1)},\alpha^{(2)})$. If $q \in \Z[\zeta_{3}]$ is primary primitive and not a perfect cube, then
\begin{align*}
 \sum_{s\equiv t \, (3)} G_K(s) \bigg[ \frac{s}{q}\bigg] \pprec \sqrt{N_{3}(q)} .
\end{align*}
\end{lemma}
\begin{proof}
By Lemma \ref{poissonlemma} we have
\begin{align*}
 \sum_{s\equiv t \, (3)} G_K(s) \bigg[ \frac{s}{q}\bigg]  &= \frac{1}{9}\sum_{\psi \in \widehat{\Z[\zeta_{3}]/3\Z[\zeta_{3}]}}\psi(-t)\sum_{\beta \, \,(3q)} \psi(\beta) \bigg[ \frac{\beta}{q}\bigg]\sum_{s \equiv \beta\,(3q)} G_K(s)  \\
& \pprec \max_{\psi,h_1,h_2} \sum_{\beta \, \,(3q)}  \bigg[ \frac{\beta}{q}\bigg] \psi_{3q}^{(h_1,h_2)}(-\beta) \psi(\beta)  \pprec \sqrt{N_{3}(q)}
\end{align*}
by the standard bound for Gauss sums on $\Z[\zeta_{3}]/3q\Z[\zeta_{3}]$ (proved by exactly the same argument as in the classical case on $\Z$). Note that the main term (corresponding to $(h_1,h_2)=(0,0)$) is 0 by orthogonality of multiplicative characters.
\end{proof}

\emph{Proof of Proposition \ref{linearsumprop}}. By Lemma \ref{twistmultiplemma}, if $q:= N_{12/3}(w)$, then for some $\omega^2+1 \equiv 0 \,\, (q)$ denoting $z=r+is$, $r,s \in \Z[\zeta_3]$
\begin{align*}
\K_w(N) \ll \sumw_{N_{12}(z) \sim N} \bigg[ \frac{s}{r}\bigg]\bigg[ \frac{r+\omega s}{q}\bigg].
\end{align*}
Shifting $s$ by $\omega r$ to get $[(r+\omega(s+r\omega))/q]=[\omega/q][s/q],$ we see by Lemma \ref{fixinglemma} that
\begin{align*}
\K_w(N) \ll \sum_{N_3(r)\, \ll\, 2\sqrt{N}} \bigg| \sum_{\substack{s \in I(r)\\ s \equiv \omega r \, (3)}} \bigg[ \frac{s}{rq}\bigg] \bigg|,
\end{align*}
where $I(r)$ is the domain in $\C$ defined by the conditions 
\begin{align*}
|r+i(s-\omega r)||r-i(s-\omega r)| \, &\sim \, N^{1/2}  \quad \quad  \quad \quad \text{and}\\
|r-i(s-\omega r)| \leq |r+i(s-\omega r)| &< |r-i(s-\omega r)| |\epsilon_0|^{12}
\end{align*}
(so that $I(r)$ is contained in the annulus $|s-r(i-\omega)| \, \asymp N^{1/4}$). If $rq$ is a perfect cube we use the trivial bound. In the remaining part we use a smooth finer-than-dyadic decomposition.
Let $K >1$ be a parameter to be optimized later. There exists a smooth partition of unity
\[
\sum_{n \in \Z} F(x-n) = 1 \quad \text{for all} \quad x \in \R
\]
for a certain fixed compactly supported $C^\infty$-smooth function $F$. By scaling and squaring we get a smooth partition
\[
\sum_{(n_1,n_2) \in \Z^2} F\bigg( \frac{x - n_1 K}{K}\bigg)F\bigg( \frac{y - n_2 K}{K}\bigg) = 1 \quad \text{for all} \quad (x,y) \in \R^2.
\].

Using this we can
 to partition $I(r)$ into smoothed boxes of side-length $ \asymp K$, obtaining $\ll N^{1/2}/K^2$ such boxes weighted with functions $G_K$ as in Lemma \ref{pvlemma}. On the boundary of $I(r)$ we use the trivial bound $K^2$ and the fact that there are $\ll N^{1/4}/K$ boxes that intersect with the boundary, so that
\begin{align*}
\K_w(N) \ll  \sum_{\substack{N_3(r) \leq 2\sqrt{N} \\ rq=t^3}}\sqrt{N}+\sum_{N_3(r) \leq 2\sqrt{N}} N^{1/4} K + \sum_{\substack{N_3(r) \leq 2\sqrt{N} \\ rq \neq t^3}}\frac{N^{1/2}}{K^2}\max_{G_K} \bigg| \sum_{s \equiv \omega r \, (3) } G_K(s) \bigg[ \frac{s}{rq}\bigg] \bigg|
\end{align*}
for any $K \ll N^{1/4}.$ By Lemma \ref{pvlemma} we have
\begin{align*}
\sum_{s \equiv \omega r \, (3) } G_K(s) \bigg[ \frac{s}{rq}\bigg] \pprec \sqrt{N_{3}(rq)} 
\end{align*}
assuming that $rq$ is not a perfect cube. Hence,
\begin{align*}
\K_w(N)& \pprec \sum_{\substack{N_3(r) \leq 2\sqrt{N} \\ rq=t^3}}\sqrt{N}+ \sum_{N_3(r) \leq 2\sqrt{N}} N^{1/4} K + \sum_{N_3(r) \leq 2\sqrt{N}}\frac{N^{1/2}}{K^2} \sqrt{N_{3}(rq)}  \\
&\pprec \sum_{\substack{n \leq 2\sqrt{N} \\ nm=t^3}}\sqrt{N}  + N^{3/4} K +  \frac{N^{5/4} N_{3}(rq)^{1/2}}{K^2} \pprec N^{2/3} + N^{3/4} K +  \frac{N^{5/4} N_{3}(q)^{1/2}}{K^2} .
\end{align*}
Choosing $K=N^{1/6} N_3(q)^{1/6} \ll N^{1/4}$ to optimize the bound we get
\begin{align*}
\K_w(N)& \pprec N^{11/12} N_3(q)^{1/6}.
\end{align*}
Note that if $K=N^{1/6} N_3(q)^{1/6} \gg N^{1/4}$ then this bound is trivial.
\qed

\section{Type II sums} \label{typeiisection}
The goal of this section is to prove Proposition \ref{typeiiprop}. It turns out that the arguments in \cite[Section 21]{FI} generalize to our case essentially verbatim. Our slightly better exponent is thanks to a small technical refinement (using a smooth weight) which is actually necessary in our situation.

By using Lemma \ref{twistmultiplemma} and the notation of Section \ref{fixingsection}, the claim is reduced to bounding (absorbing $[z]$ and $[w]$ into the coefficients)
\begin{align*}
Q(M,N):= \sumw_{N_{12}(w) \sim M} \, \, \sumw_{N_{12}(z) \sim N} \alpha_w \beta_z \bigg(\frac{z}{w} \bigg),
\end{align*}
where $\alpha_w$ and $\beta_z$ are supported on primitive numbers. First we require the following preliminary bound for $Q(M,N)$ (compare to \cite[Lemma 21.2]{FI}).
\begin{lemma} \label{prelimtypeiilemma} We have
\begin{align*}
Q(M,N) \pprec M^{5} +M^{1/2} N.
\end{align*}
\end{lemma}
\begin{proof}
We may assume that $N>  M^{4+\eta}$ for some small $\eta>0$, since otherwise by a trivial bound
\begin{align*}
Q(M,N) \pprec MN \leq M^{5+\eta} .
\end{align*}

By Cauchy-Schwarz and by Lemma \ref{fixinglemma} we have
\begin{align*}
Q(M,N) &\pprec  N^{1/2} \bigg( \sumw_{N_{12}(z) \sim N}  \bigg| \sumw_{N_{12}(w) \sim M}\alpha_w \bigg(\frac{z}{w} \bigg) \bigg|^2 \bigg)^{1/2} \\ 
&\leq N^{1/2} \bigg( \sumw_{w_1,w_2} \alpha_{w_1} \overline{\alpha}_{w_2}  \sum_{\substack{z=r+is \\ r \equiv 1 \, (3), \, 3|s}} F_{\sqrt{N}}(r) F_{\sqrt{N}}(s) \bigg(\frac{z}{w_1} \bigg) \bigg(\frac{z}{w_2} \bigg)^2 \bigg)^{1/2} 
\end{align*}
where  $F_{\sqrt{N}}(r)= F(N_3(r)/\sqrt{N})$ for a fixed compactly supported $C^\infty$-smooth function $F$ so that $F_{\sqrt{N}}(r) F_{\sqrt{N}}(s)$ is a smooth majorant for the original summation range. Let $q:= N_{12/3}(w_1w_2) = u^2+v^2$ with $(u,v)=1$. Note that $(3,q)=1$ since $-1$ is not a square in $\Z[\zeta_3]/(1-\zeta_3)\Z[\zeta_3].$  Splitting the sum over $z$ into residue classes modulo $3q$ we get for some constant $c_0$
\begin{align*}
\sum_{\substack{z=r+is \\ r \equiv 1 \, (3), \, 3|s}} & F_{\sqrt{N}}(r) F_{\sqrt{N}}(s) \bigg(\frac{z}{w_1} \bigg) \bigg(\frac{z}{w_2} \bigg)^2  \\
&= \sum_{\substack{\zeta = r_0+i s_0 \in \Z[\zeta_{12}]/3q\Z[\zeta_{12}] \\ r_0 \equiv 1 \, (3), \,3| s_0}} \bigg(\frac{\zeta}{w_1} \bigg) \bigg(\frac{\zeta}{w_2} \bigg) ^2 \sum_{\substack{r,s \in \Z[\zeta_3]  \\ (r,s) \equiv (r_0,s_0) \,\, (3q)} }  F_{\sqrt{N}}(r) F_{\sqrt{N}}(s) \\ &= \bigg(\frac{c^2_0 N}{N_{3}(3q)^2} +O_C\bigg(X^{- C}\bigg)\bigg)\sum_{\substack{\zeta = r_0+i s_0 \in \Z[\zeta_{12}]/3q\Z[\zeta_{12}] \\ r_0 \equiv 1 \, (3), \,3| s_0}} \bigg(\frac{\zeta}{w_1} \bigg) \bigg(\frac{\zeta}{w_2} \bigg) ^2
\end{align*}
by Lemma \ref{poissonlemma} since $N_3(q) \ll M^2 < X^{-\eta} \sqrt{N}$. Recall the notation of Lemma \ref{symbolmultiplemma}, that is,
\begin{align*}
q_j= N_{12/3}(w_j), \quad e:=(w_1,\sigma(w_2^2)),\quad \text{and} \quad d:= N_{12/3}(e),
\end{align*}
where $\sigma$ is the conjugation $\sigma(a+ib)=a-ib$. 
 By Lemma \ref{symbolmultiplemma} we have (writing $x:=r_0-\omega s_0$, $y:= r_0+\omega s_0$)
\begin{align*}
 \sum_{\substack{\zeta = r_0+i s_0 \in \Z[\zeta_{12}]/3q\Z[\zeta_{12}] \\ r_0 \equiv 1 \, (3), \,3| s_0}}  & \bigg(\frac{\zeta}{w_1} \bigg) \bigg(\frac{\zeta}{w_2} \bigg) ^2 = \sum_{\substack{r_0,s_0 \in \Z[\zeta_{3}]/3q\Z[\zeta_{3}] \\ r_0 \equiv 1 \, (3), \,3| s_0}} \bigg[\frac{r_0-\omega s_0}{d} \bigg] \bigg[\frac{r_0+\omega s_0}{q_1q_2^2/d} \bigg] \\
&= \sum_{\substack{x \in \Z[\zeta_{3}]/3q\Z[\zeta_{3}] \\ x \equiv 1 \, (3)}}\bigg[\frac{x}{d}\bigg] \sum_{\substack{y \in \Z[\zeta_{3}]/3q\Z[\zeta_{3}] \\ y \equiv 1 \, (3)}} \bigg[\frac{y}{q_1q_2^2/d} \bigg]  \\
&= N_3(q)\sum_{\substack{x \in \Z[\zeta_{3}]/3d\Z[\zeta_{3}] \\ x \equiv 1\, (3)}}\bigg[\frac{x}{d}\bigg] \sum_{\substack{y \in \Z[\zeta_{3}]/(3q/d)\Z[\zeta_{3}] \\ y \equiv 1 \, (3)}} \bigg[\frac{y}{q_1q_2^2/d} \bigg] = 0
\end{align*}
unless both of $d$ and $q_1q_2^2/d$ are perfect cubes in which case we get
\begin{align*}
N_3(q) \varphi(N_3(d))\varphi(N_3(q/d)) \leq N_{3}(q)^2.
\end{align*} 
To see this, recall that $(3,q)=1$ and note that $[y/(q_1q_2^2/d)]$ is a cubic character modulo $q/d$. Hence, we get
\begin{align*}
Q(M,N)^2 \pprec N^2 \sum_{\substack{q_1,q_2 \in \Z[\zeta_{3}] \\ q_1q_2^2 = t^3 \\ N_3(q_1), N_3(q_2) \leq 2M }}1.
\end{align*}
We have
\begin{align*}
\sum_{\substack{q_1,q_2 \in \Z[\zeta_{3}] \\ q_1q_2^2 = t^3 \\ N_3(q_1), N_3(q_2) \leq 2M }} 1 \pprec \sum_{\substack{m_1, m_2 \leq 2M \\ m_1m_2^2=c^3}} 1 = \sum_{d \leq 2M} \sum_{\substack{m_1,m_2 \leq 2M/d \\ (m_1,m_2)=1 \\ m_1=c_1^3, \, m_2=c_2^3}} 1 \ll M^{2/3} \sum_{d \leq 2M} d^{-2/3} \ll M,
\end{align*}
so that for $N> X^\eta M^4$
\begin{align*}
Q(M,N) \pprec M^{1/2} N .
\end{align*}
\end{proof}
\begin{remark} Note that the modulus $q_1q_2^2$ being a cube is morally the same as $q_1 q_2$ being a square which explains why we get the same term $M^{1/2}N$ as in \cite[Lemma 21.2]{FI}.
\end{remark}
Lemma \ref{prelimtypeiilemma} is non-trivial as soon as $N \gg M^4 \gg 1$. Similarly as in \cite{FI}, we now use H\"older's inequality to extend this range so that we can handle the range $M \asymp N$. We get
\begin{align*}
Q^k(M,N) \pprec M^{k-1} \sumw_{w} \bigg| \sumw_z \beta_z \bigg( \frac{z}{w}\bigg)\bigg|^k =M^{k-1} \tilde{Q}(M,N^k),
\end{align*}
where $\tilde{Q}$ is of  similar form as $Q$ except that $\alpha_w$ is replaced by some coefficients $\tilde{\alpha}_w$ with $|\tilde{\alpha}_w| \leq 1$ and $\beta_z$ is replaced by the divisor bounded coefficient
\begin{align*}
\tilde{\beta}_z = \sumw_{z_1\cdots z_k=z} \beta_{z_1} \cdots \beta_{z_k}
\end{align*}
where now each $z_1$ satisfies the conditions of Lemma \ref{fixinglemma} separately. Applying Lemma \ref{prelimtypeiilemma} to $\tilde{Q}(M,N^k)$ we get
\begin{align*}
Q(M,N) \pprec M^{1+4/k} + M^{1-1/2k} N.
\end{align*}
Similarly as in \cite{FI}, we will use reciprocity below to get a symmetric bound. Hence, the most difficult range will be $M=N$. To optimize the bound we choose $k=5$  to get
\begin{align*}
Q(M,N) \pprec M^{9/5} + M^{9/10}N \pprec M^{9/10}N
\end{align*}
if $M \leq N$.  Since by Lemma \ref{twistmultiplemma} $(z/w)=(w/z)$ for primitive primary $z$ and $w$, the form $Q(M,N)$ is symmetric and we get
\begin{align*}
Q(M,N) \pprec MN^{9/10} + M^{9/10} N.
\end{align*} \qed
\begin{remark} The reason our bound is superior to that in \cite[Proposition 21.3]{FI} is that we used the smooth function $F_{\sqrt{N}}(r)$ instead of a sharp cut-off $N_3(r)\leq \sqrt{N}$ in the proof of Lemma \ref{prelimtypeiilemma}. This allows us to improve the term $M^2N^{3/4}$ appearing in Lemma \cite[Lemma 21.2]{FI} to $M^5$. Obviously the same refinement can be implemented to improve their bound to the same form.
\end{remark}

\section{Connection to primes of the form $\alpha^2+\beta^6$ on $\Z[\zeta_3]$} \label{primessection}
It is tempting to ask if the method of Friedlander and Iwaniec can be extended to produce primes of the form $a^2+b^6$ on $\Z$ (cf. \cite[Remarque 4.20]{michel}). Unfortunately there seems to be two large obstacles to this. Firstly, the sequence is too sparse for replicating the steps in \cite[Sections 5-9]{FI}. The second problem is structural -- the proofs in \cite{FI} rely on the law of quadratic reciprocity in multiple  places, while for cubic residues we do not have a suitable reciprocity law on $\Z$. To mend this we need to transfer the whole set-up to $\Z[\zeta_3]$. Unfortunately the first problem persist (cf. the paragraph around (9.4) below).

In this section we explain how the sum in Theorem \ref{maintheorem} arises if we consider primes of the form $\alpha^2+\beta^6$ on $\Z[\zeta_3]$, which was the original motivation for this manuscript. All of the discussion presented here is non-rigorous and for the sake of illustration we omit all of the technical issues that arise in \cite{FI}. The argument follows exactly the same lines as in \cite{FI}.

Note that if $z \in \Z[\zeta_3]$ is the sum of two squares, it has infinitely many such representations. To make the analogy with the Friedlander-Iwaniec Theorem precise we consider primary primes $\pi=\beta^3+i\alpha$ on $\Z[\zeta_{12}]$ with the restriction 
\begin{align} \label{genrestriction}
N_{12}(\beta^3+i\alpha)^{1/4} \leq |\beta^3+i\alpha| < N_{12}(\beta^3+i\alpha)^{1/4} |\epsilon_0|^6
\end{align}
as in Section \ref{fixingsection}. Note that this implies that $|\alpha| \ll x^{1/4}$ and  $|\beta| \, \ll x^{1/12}$ if $N_{12}(\pi) \ll x$.

The number of ideals $\a$ with $N_{12} \a \leq x$ that have such a generator $\beta^3+i\alpha$ is 
\begin{align*}
\sum_{N_{12} \a \leq x} \,\,\sumw_{\substack{y=\beta^3+i\alpha \\ (y)=\a}} 1 \asymp x^{2/3},
\end{align*}
which is very sparse (the inability of the Friedlander-Iwaniec method to handle sets of density $<x^{-1/3}$ is also noted by Helfgott \cite{helfgott}). Due to this we are not able to handle the Type II sums, as we shall see soon below (see (\ref{lengthproblem})). However, it should be possible to obtain an approximation to this problem by considering primes of the form $\alpha^2+\lambda^2 \beta^6$ where $\lambda$ runs over elements of small norm $N_{3}(\lambda) \leq x^{\delta}$, to show that for some fairly small $\delta>0$ we get a lower bound of the correct order of magnitude for the number of such primes. For this the sieve of Friedlander and Iwaniec needs to be replaced by Harman's sieve (see  \cite{maynard} for a version of this on Number fields).

\begin{remark}
Restricting to generators satisfying (\ref{genrestriction}) does not decrease the density essentially. Most of the integers $\beta^3+i\alpha \in \Z[\zeta_{12}]$ with $|\alpha^2 + \beta^6|^2 \leq x$ come from the part $|\alpha| \ll x^{1/4}$ and $|\beta| \ll x^{1/12}$. To see this note that we have $|\beta^3+i\alpha||\beta^3-i\alpha| \leq x^{1/2}$. Suppose that $|\beta^3+i\alpha|$ and $|\beta^3-i\alpha|$ are not of similar size, say,  $|\beta^3-i\alpha| \asymp Y$ for some $1 \ll Y \ll x^{1/4}$. For any given $\beta$ there are roughly $Y^2$ choices of $\alpha$ that satisfy $|\beta^3-i\alpha| \asymp Y$. But then $|\beta^3+i\alpha| \ll x^{1/2}/Y$ is morally the same as $|\beta| \ll x^{1/6}/Y^{1/3}$, so that there are roughly $x^{1/3}/Y^{2/3}$ choices for $\beta$ and we get
\begin{align*}
|\{\alpha,\beta \in \Z[\zeta_3]: |\alpha^2 + \beta^6|^2 \leq x, \, |\beta^3-i\alpha| \asymp Y \}| \asymp Y^2 \cdot x^{1/3}/Y^{2/3} = x^{1/3} Y^{4/3},
\end{align*}
which is much less than $x^{2/3}$ if $Y$ is much smaller than $x^{1/4}$. Heuristically the same holds also for $Y <1$, since the probability of finding a lattice point $i\alpha$ near $\beta^3$ is proportional to the area $Y^2$. Thus, we cannot increase the density by considering all generators $\beta^3+i\alpha$ instead of just those satisfying (\ref{genrestriction}).
\end{remark}

Using a sieve argument the main problem is to handle Type II sums of the form
\begin{align*}
S_1 := \sum_{N_{12} \m \sim M} \alpha_\m \sum_{N_{12} \n \sim N} \beta_\n \,\,\sumw_{\substack{y=\beta^3+i\alpha \\ (y)=\m\n}} 1,
\end{align*}
where $MN=x$, and $\alpha$ and $\beta$ are bounded coefficients with $\beta$ behaving like a M\"obius function in terms of a Siegel-Walfisz type condition. The goal then is to show that $S_1 \ll_C x^{2/3} \log^{-C} x$.

We now pick a primary generator  $z$ of $\n$ according to Section \ref{fixingsection}, and then pick a generator $\sigma(w)$ of $\m$ such that $y=\sigma(w) z$. The cross-condition
\begin{align} \label{ccwz}
N_{12}(w)^{1/4} N_{12}(z)^{1/4}  \leq |z||w| < N_{12}(w)^{1/4} N_{12}(z)^{1/4}  |\epsilon_0|^6
\end{align}
is easily removed by Perron's formula, so that we are essentially left with
\begin{align*}
S_2 := \sumw_{N_{12} (w )\sim M} \alpha_{\sigma(w)} \sumw_{N_{12}( z )\sim N} \beta_{z} \,\,\sum_{\substack{\sigma(w) z=\beta^3+i\alpha }} 1
\end{align*}
(note that (\ref{ccwz}) together with $|z| \asymp N_{12}(z)^{1/4}$ implies $|w| \asymp N_{12}(w)^{1/4}$ which is morally same as the  condition implied by the $\wedge$ in the sum over $w$).

Applying the Cauchy-Schwarz inequality similarly as in \cite{FI} we get $S_2 \ll M^{1/2} S_3^{1/2}$ for
\begin{align*}
S_{3} := \sumw_{z_1, z_2} \beta_{z_1} \overline{\beta}_{z_2} \sum_{w} \sum_{\substack{\sigma(w) z_1=\beta_1^3+i\alpha_1 \\  \sigma(w) z_2=\beta_2^3+i\alpha_2}} F_M(w)
\end{align*}
for some suitable smooth function $F_M$ supported on $N_{12}(w) \asymp M$ and $|w| \asymp M^{1/4}$.

From the diagonal part $z_1=z_2$ we get a contribution $(MN)^{2/3}$, which is sufficient for $S_1 \ll x^{2/3} \log^{-C} x$ as long as 
\begin{align} \label{diagonal}
N \gg x^{1/3} \log^C x.
\end{align}

In the off-diagonal the generic case is $(z_1,z_2)=1$, and by the same argument as in \cite[Section 6]{FI} we have
\begin{align*}
i \Delta w = \beta_1^3 z_2 - \beta_2^3 z_1,
\end{align*}
where $z_j =r_j+is_j$ and $\Delta := r_1 s_2-r_2 s_1$. Thus, the off-diagonal part is
\begin{align*}
S_4 &:= \sumw_{(z_1, z_2)=1} \beta_{z_1} \overline{\beta}_{z_2} \sum_{\beta_1^3 z_2 \equiv \beta_2^3 z_1 \,\, (\Delta)} F_M(( \beta_1^3 z_2 - \beta_2^3 z_1)/i\Delta) \\
&= \sumw_{(z_1, z_2)=1} \beta_{z_1} \overline{\beta}_{z_2} \sum_{\substack{\gamma_1, \gamma_2 \,(\Delta) \\ \gamma_1^3 z_2 \equiv \gamma_2^3 z_1 \,\, (\Delta)}} \sum_{(\beta_1,\beta_2) \equiv (\gamma_1,\gamma_2) \, (\Delta)} F_M(( \beta_1^3 z_2 - \beta_2^3 z_1)/i\Delta).
\end{align*}
Similarly to \cite{FI} we note that the congruence $\beta_1^3 z_2 \equiv \beta_2^3 z_1 \,\, (\Delta)$ is in fact a $\Z[\zeta_3]$-rational congruence, since
\begin{align*}
z_2 z_1^{-1} \equiv \frac{r_1r_2+s_1s_2}{r_1^2+s_1^2} \,\, (\Delta).
\end{align*}

Unfortunately here we run into a problem with the sparseness of our sequence (note that also in \cite{FI} for the Type II sums this is the only part of the argument affected by the sparsity). We would like to apply Poisson summation (Lemma \ref{poissonlemma}) to evaluate the smoothed sum over $\beta_1,\beta_2$. However, due to the diagonal contribution (\ref{diagonal}) we have essentially
\begin{align*}
N_{3}(\Delta) \approx N \gg x^{1/3} \log^C x
\end{align*}
while the length of the sum is
\begin{align} \label{lengthproblem}
N_{3}(\beta_j) \ll x^{1/6}
\end{align}
which is just narrowly too short (Poisson summation becomes ineffective if the length of the sum is less than square root of the size of the modulus). Due to this we are not able to evaluate the sum in any range of $N$. If we consider the aforementioned approximation version of the problem with $\alpha^2+\lambda^2 \beta^6$ for $N_{3}(\lambda) \leq x^{\delta}$, then the diagonal part gives a restriction $N \gg x^{1/3-2\delta/3} \log^C x$ which gives some room to work with. We have been able to evaluate the sum over  $\beta_1^3 \lambda_1 z_2 \equiv \beta_2^3 \lambda_2 z_1 \,\, (\Delta)$ in some ranges using a large sieve argument similar to that in Heath-Brown and Li \cite{hbli} (although the argument required here is much more intricate).

\begin{remark} Since the argument falls short barely, there is some hope that with a delicate estimate we could handle Type II sums in some very narrow non-trivial range for $N=x^{1/3+o(1)}$. This would suffice to break the parity barrier, that is, to show that there are infinitely many of $\alpha^2+\beta^6$ are a product of exactly two primes (one of size $M$ and the other of size $N$).
\end{remark}

Assuming that the sum over $\beta_1,\beta_2$ could be computed, then the main term is essentially (up to a multiplicative factor and a smooth coefficient, and ignoring the fact that $\gamma_j$ may have common factors with $\Delta$)
\begin{align*}
S_5:= \sumw_{(z_1, z_2)=1} \beta_{z_1} \overline{\beta}_{z_2} \sum_{\substack{(\omega,\Delta)=1 \\\omega^3 \equiv z_2 z_1^{-1} \, (\Delta)}} 1
\end{align*}
(for the approximate version of the problem the congruence is $\omega^3 \equiv z_2 \lambda_1 z_1^{-1} \lambda_2^{-1} \, (\Delta)$, which is morally the same). Here we need to show only a little bit of cancellation, that is, $S_5 \ll N^2 \log^{-C}x$. To evaluate the sum over cubic roots we make use of the Chinese Reminder Theorem and the cubic residue character to get (assuming $\Delta$ is square-free, primitive, and ignoring the fact that $3|\Delta$)
\begin{align*}
\sum_{\substack{(\omega,\Delta)=1 \\\omega^3 \equiv z_2 z_1^{-1} \, (\Delta)}} 1 &= \prod_{\pi | \Delta} \bigg( 1+ \bigg[ \frac{z_2 z_1^{-1}}{\pi} \bigg] +  \bigg[ \frac{z_2 z_1^{-1}}{\pi} \bigg]^2 \bigg) \\
&=\sum_{\delta_1 \delta_2 | \Delta}  \bigg[\frac{z_2 z_1^{-1}}{\delta_1} \bigg] \bigg[\frac{z_2 z_1^{-1}}{\delta_2} \bigg]^2.
\end{align*}
Thus, we essentially get (compare to $T (\beta)$ in \cite[Section 10]{FI})
\begin{align*}
S_5 = \sum_{\delta_1,\delta_2} \,\, \sumw_{\substack{(z_1, z_2)=1 \\ \Delta \equiv 0 \,(\delta_1 \delta_2)}} \beta_{z_1} \overline{\beta}_{z_2} \bigg[\frac{z_2 z_1^{-1}}{\delta_1} \bigg] \bigg[\frac{z_2 z_1^{-1}}{\delta_2} \bigg]^2.
\end{align*}
Similarly as in \cite[Section 10]{FI}, we now split the sum into three parts $U+V+W$ according to the size of $\delta_1\delta_2$, where in $U$ we have $N_3(\delta_1\delta_2) \ll \log^C N$, in $W$ we have $N_{3}(\delta_1 \delta_2) \gg N_{3}(\Delta) \log^{-C}$, and $V$ is the remaining middle part. 

For $U$ we get the required cancellation from the $\beta_{z}$, which look like a M\"obius function, by using a suitable Siegel-Walfisz type bound.

For $V$ we have not checked in detail but we expect that the large sieve -type arguments in \cite[Sections 11-15]{FI} generalize to our case.

For $W$ the generic case is $\delta_1 \delta_2 = \Delta$. To handle this we need the following analogue of \cite[Lemma 17.1]{FI}, which we will prove at the end of this section.
\begin{lemma} \label{splittinglemma} For $z_1,z_2, \Delta=r_1s_2-r_2s_1$ with $z_1,z_2$ primary and $z_1\equiv z_2 \,\, (9)$ we have
\begin{align*}
\bigg[\frac{z_2 z_1^{-1}}{\Delta} \bigg] = \bigg[\frac{s_1}{r_1}\bigg]^2 \bigg[\frac{s_2}{r_2}\bigg].
\end{align*}
\end{lemma}
To guarantee that $z_1 \equiv z_2 \, (9)$ we have to split $z$ into residue classes modulo 9 before the application of Cauchy-Schwarz.

Using this lemma the sum $W$ is essentially reduced to
\begin{align*}
S_6 &:=  \sumw_{\substack{(z_1, z_2)=1 }} \beta_{z_1} \overline{\beta}_{z_2} \sum_{\Delta=\delta_1\delta_2}  \bigg[\frac{z_2 z_1^{-1}}{\delta_1} \bigg] \bigg[\frac{z_2 z_1^{-1}}{\delta_2} \bigg]^2 \\
&= \sumw_{\substack{(z_1, z_2)=1 }} \beta_{z_1} \overline{\beta}_{z_2}  \bigg[\frac{z_2 z_1^{-1}}{\Delta} \bigg] \sum_{\delta_2 | \Delta}  \bigg[\frac{z_2 z_1^{-1}}{\delta_2} \bigg] \\
&= \sumw_{\substack{(z_1, z_2)=1 }} \beta_{z_1} \bigg[\frac{s_1}{r_1}\bigg]^2 \overline{\beta}_{z_2}   \bigg[\frac{s_2}{r_2}\bigg] \sum_{\delta_2 | \Delta}  \bigg[\frac{z_2 z_1^{-1}}{\delta_2} \bigg].
\end{align*}

 We now again partition the sum into three parts $W_1+W_2+W_3$ according to the size of $\delta_2$.

In $W_1$ the generic case is $\delta_2 =1$ and we get a sum
\begin{align*}
\sumw_{\substack{(z_1, z_2)=1 }} \beta_{z_1} \bigg[\frac{s_1}{r_1}\bigg]^2 \overline{\beta}_{z_2}   \bigg[\frac{s_2}{r_2}\bigg],
\end{align*}
which can be bounded using similar arguments as in the proof of Theorem \ref{maintheorem} (once we remove the condition $(z_1, z_2)=1$ either by M\"obius function or by the device in \cite{FI}), since the coefficients $\beta_z$ have a Type I/Type II decomposition.

Similarly for $W_3$ the generic case is $\delta_2= \Delta$ and we get by a second application of Lemma \ref{splittinglemma} a sum
\begin{align*}
\sumw_{\substack{(z_1, z_2)=1 }} \beta_{z_1} \bigg[\frac{s_1}{r_1}\bigg] \overline{\beta}_{z_2}   \bigg[\frac{s_2}{r_2}\bigg]^2
\end{align*}
which is of the same form as with $W_1$.

Finally for $W_2$ we note that this sum is essentially contained in the earlier sum $V$ but with the coefficients  $\beta_z$ twisted by $[z]$ or $[z]^2$, so that the same argument as with $V$ should take care of this part.

\emph{Proof of Lemma \ref{splittinglemma}.}
The argument is essentially the same as the proof of \cite[Lemma 17.1]{FI}. For simplicity we give the proof only in the case $(r_1,r_2)=1$ (in the general case we have to juggle back and forth with the factor $(r_1,r_2)$). Since $[r_1^3 / \Delta]=1$, we have
\begin{align*}
\bigg[\frac{z_2 z_1^{-1}}{\Delta} \bigg] =\bigg[\frac{(r_1^2+s_1^2)^{-1}(r_1r_2+s_1s_2)}{\Delta} \bigg]  =\bigg[\frac{(r_1^2+s_1^2)^{-1}(r_1r_2+s_1s_2)r_1^3}{\Delta} \bigg].
\end{align*}
We have
\begin{align*}
(r_1r_2+s_1s_2)r_1^3 - (r_1^2+s_2^2)r_1^2r_2 = r_1^2s_1(r_1s_2-r_2s_1) \equiv 0 \,\, (\Delta),
\end{align*}
so that
\begin{align*}
(r_1^2+s_1^2)^{-1}(r_1r_2+s_1s_2)r_1^3 \equiv r_1^2r_2 \, \,(\Delta).
\end{align*}
Hence,
\begin{align*}
\bigg[\frac{z_2 z_1^{-1}}{\Delta} \bigg]=\bigg[\frac{r_1^2r_2}{\Delta} \bigg] = \bigg[\frac{r_1}{\Delta} \bigg]^2\bigg[\frac{r_2}{\Delta} \bigg] = \bigg[\frac{\Delta}{r_1} \bigg]^2\bigg[\frac{\Delta}{r_2} \bigg]
\end{align*}
by the supplementary laws in Lemma \ref{cubicreciprocitylemma}, since $z_1 \equiv z_2 \,\, (9)$. Therefore,
\begin{align*}
\bigg[\frac{z_2 z_1^{-1}}{\Delta} \bigg] =\bigg[\frac{r_1s_2-r_2s_1}{r_1} \bigg]^2\bigg[\frac{r_1s_2-r_2s_1}{r_2} \bigg] = \bigg[\frac{-r_2s_1}{r_1} \bigg]^2\bigg[\frac{r_1s_2}{r_2} \bigg] = \bigg[\frac{s_1}{r_1} \bigg]^2\bigg[\frac{s_2}{r_2} \bigg],
\end{align*}
since by Lemma \ref{cubicreciprocitylemma} for $r_1$ and $r_2$ primary
\begin{align*}
\bigg[\frac{-r_2}{r_1} \bigg]^2\bigg[\frac{r_1}{r_2} \bigg]=\bigg[\frac{r_2}{r_1}  \bigg]^2\bigg[\frac{r_1}{r_2} \bigg]=\bigg[\frac{r_1}{r_2} \bigg]^3 =1.
\end{align*}
\qed

\bibliography{cubicspinbibl}
\bibliographystyle{abbrv}

\end{document}